%% file: PhaseMaxcut.tex
\documentclass[11pt,reqno]{amsart}
\usepackage{fullpage,times,graphicx,amssymb,amsmath,amsfonts,psfrag,xcolor,pdfsync}
\input defs.tex

\input defs_p.tex
\usepackage[colorlinks,citecolor=bbluegray,linkcolor=ddarkbrown,urlcolor=blue,breaklinks]{hyperref}

\usepackage{auto-pst-pdf} 
\begin{document}
\title{Phase Recovery,  MaxCut\\ and Complex Semidefinite Programming}

\author{Ir\`ene Waldspurger}
\address{D.I., \'Ecole Normale Sup\'erieure, Paris.}
\email{waldspur@clipper.ens.fr}

\author{Alexandre d'Aspremont} 
\address{CNRS \& D.I., \'Ecole Normale Sup\'erieure, Paris. UMR 8548.}
\email{alexandre.daspremont@m4x.org}

\author{St\'ephane Mallat}
\address{D.I., \'Ecole Normale Sup\'erieure, Paris.}
\email{mallat@cmap.polytechnique.fr}

\date{\today}

\begin{abstract}
Phase retrieval seeks to recover a signal $x \in \complex^p$ from the amplitude $|A x|$ of linear measurements~$Ax \in \complex^n$. We cast the phase retrieval problem as a non-convex quadratic program over a complex phase vector and formulate a tractable relaxation (called \ref{eq:ph-SDP}) similar to the classical {\em MaxCut} semidefinite program. We solve this problem using a provably convergent block coordinate descent algorithm whose structure is similar to that of the original greedy algorithm in \cite{gerchberg}, where each iteration is a matrix vector product. Numerical results show the performance of this approach over three different phase retrieval problems, in comparison with greedy phase retrieval algorithms and matrix completion formulations.
\end{abstract}
\maketitle

\section{Introduction}
The phase recovery problem, i.e. the problem of reconstructing a complex phase vector given only the magnitude of linear measurements, appears in a wide range of engineering and physical applications. It is needed for example in X-ray and crystallography imaging \citep{Harr93}, diffraction imaging \citep{Bunk07} or microscopy \citep{Miao08}. In these applications, the detectors cannot measure the phase of the incoming wave and only record its amplitude. Recovering the complex phase of wavelet transforms from their amplitude also has applications in audio signal processing \citep{Grif84}.

In all these problems, complex measurements of a signal $x \in \complex^p$ are obtained from a linear injective operator $A$, but we can only measure the magnitude vector $|A x|$. Depending on the properties of $A$, the phase of $A x$ may or may not be uniquely characterized by the magnitude vector $|A x|$, up to an additive constant, and it may or may not be stable. For example, if $A$ is a one-dimensional Fourier transform, then the recovery is not unique but it becomes unique almost everywhere for an oversampled two-dimensional Fourier transform, although it is not stable. Uniqueness is also obtained with an oversampled wavelet transform operator $A$, and the recovery of $x$ from $|Ax|$ is then continuous \citep{Wald12}. If $A$ consists in random Gaussian measurements then uniqueness can be proved, together with stability results \citep{Cand11, Cand12}.

Recovering the phase of $Ax$ from $|A x|$ is a nonconvex optimization problem. Until recently, this problem was solved using various greedy algorithms \citep{gerchberg,Fien82,Grif84}, which alternate projections on the range of $A$ and on the nonconvex set of vectors $y$ such that $|y| = |A x|$. However, these algorithms often stall in local minima. A convex relaxation called \ref{eq:ph-lift} was introduced in \citep{Chai11} and \citep{Cand11} by observing that $|A x|^2$ is a linear function of $X = x x^*$ which is a rank one Hermitian matrix. The recovery of $x$ is thus expressed as a rank minimization problem over positive semidefinite Hermitian matrices $X$ satisfying some linear conditions. This last problem is approximated by a semidefinite program. It has been shown that this program recovers $x$ with high probability when $A$ has gaussian independant entries \citep{Cand11,Cand11a}. Numerically, the same result seems to hold for several classes of linear operators~$A$.

Our main contribution here is to formulate phase recovery as a quadratic optimization problem over the unit complex torus. We then write a convex relaxation to phase recovery very similar to the {\em MaxCut} semidefinite program (we call this relaxation \ref{eq:ph-SDP} in what follows). While the resulting SDP is typically larger than the \ref{eq:ph-lift} relaxation, its simple structure (the constraint matrices are singletons) allows us to solve it very efficiently. In particular, this allows us to use a provably convergent block coordinate descent algorithm whose iteration complexity is similar to that of the original greedy algorithm in \cite{gerchberg} (each iteration is a matrix vector product, which can be computed efficiently). We also show that tightness of \ref{eq:ph-lift} implies tightness of a modified version of \ref{eq:ph-SDP}. Furthermore, under the condition that $A$ is injective and $b$ has no zero coordinate, we derive an equivalence result between \ref{eq:ph-lift} and a modified version of \ref{eq:ph-SDP}, in the noiseless setting. This result implies that both algorithms are simultaneously tight (an earlier version of this paper showed \ref{eq:ph-lift} tightness implies \ref{eq:ph-SDP} tightness and the reverse direction was then proved in \citep{Voro12a} under mild additional assumptions). In a noisy setting, one can show that \ref{eq:ph-SDP} is at least as stable as a variant of \ref{eq:ph-lift}, while~\ref{eq:ph-SDP} empirically appears to be more stable in some cases, e.g. when $b$ is sparse.

Seeing the {\em MaxCut} relaxation emerge in a phase recovery problem is not entirely surprising: it appears, for example, in an angular synchronization problem where one seeks to reconstruct a sequence of angles~$\theta_i$ (up to a global phase), given information on pairwise differences $\theta_i-\theta_j$ mod. $2\pi$, for $(i,j)\in S$ \citep[see][]{Sing11}, the key difference between this last problem and the phase recovery problem in~\eqref{eq:ph-recov} is that the sign information is lost in the input to~\eqref{eq:ph-recov}. % Mention Log K SDP approximation in Nemirovski and Man Cho So).
Complex {\em MaxCut}-like relaxations of decoding problems also appear in maximum-likelihood channel detection \citep{Luo03,Kisi10,So10}. From a combinatorial optimization perspective, showing the equivalence between phase recovery and {\em MaxCut} allows us to expose a new class of nontrivial problem instances where the semidefinite relaxation for a {\em MaxCut}-like problem is tight, together with explicit conditions for tightness directly imported from the matrix completion formulation of these problems (these conditions are of course also hard to check, but hold with high probability for some classes of random experiments).

The paper is organized as follows. Section \ref{s:problem} explains how to factorize away the magnitude information to form a nonconvex quadratic program on the phase vector $u \in \complex^n$ satisfying $|u_i| = 1$ for $i=1,\ldots,n$, and a greedy algorithm is derived in Section~\ref{s:greedy}. We then derive a tractable relaxation of the phase recovery problem, written as a semidefinite program similar to the classical {\em MaxCut} relaxation in \citep{goem95}, and detail several algorithms for solving this problem in Section~\ref{s:algos}. Section~\ref{s:completion} proves that a variant of \ref{eq:ph-SDP} and \ref{eq:ph-lift} are equivalent in the noiseless case and thus simultaneously tight. We also prove that \ref{eq:ph-SDP} is as stable as a weak version of \ref{eq:ph-lift} and discuss the relative complexity of both algorithms. Finally, Section \ref{s:numres} performs a numerical comparison between the greedy, \ref{eq:ph-lift} and \ref{eq:ph-SDP} phase recovery algorithms for three phase recovery problems, in the noisy and noiseless case. In the noisy case, these results suggest that if $b$ is sparse, then \ref{eq:ph-SDP} may be more stable than \ref{eq:ph-lift}.

\subsection*{Notations}
We write $\symm_p$ (resp. $\herm_p$) the cone of symmetric (resp. Hermitian) matrices of dimension $p$ ; $\symm_p^+$ (resp. $\herm_p^+$) denotes the set of positive symmetric (resp. Hermitian) matrices. We write $\|\cdot\|_p$ the Schatten $p$-norm of a matrix, that is the $p$-norm of the vector of its eigenvalues (in particular, $\|\cdot\|_\infty$ is the spectral norm). We write $A^\dag$  the (Moore-Penrose) pseudoinverse of a matrix $A$ and $\|A\|_{\ell_1}$ the sum of the modulus of the coefficients of $A$. For $x\in\reals^p$, we write $\diag(x)$ the matrix with diagonal $x$. When $X\in\herm_p$ however, $\diag(X)$ is the vector containing the diagonal elements of $X$. For $X\in\herm_p$, $X^*$ is the {H}ermitian transpose of $X$, with $X^*=(\bar X)^T$. Finally, we write $b^2$ the vector with components $b_i^2$, $i=1,\ldots,n$.

\vskip 6ex
\section{Phase recovery}\label{s:problem}
The phase recovery problem seeks to retrieve a signal $x \in \complex^p$ from the amplitude $b=|A x|$ of $n$ linear measurements, solving 
\BEQ\label{eq:ph-recov}
\BA{ll}
\mbox{find} & x\\
\mbox{such that} & |Ax|= b,
\EA\EEQ
in the variable $x\in\complex^p$, where $A\in\complex^{n \times p}$ and $b\in\reals^n$.

\subsection{Greedy Optimization in the Signal}\label{s:greedy-sig}
Approximate solutions $x$ of the recovery problem in~\eqref{eq:ph-recov} are usually computed from $b = |A x|$ using algorithms inspired from the alternating projection method in \citep{gerchberg}. These algorithms compute iterates $y^k$ in the set $\bf F$ of vectors $y\in\complex^n$ such that $|y| = b = |Ax|$, which are getting progressively closer to the image of $A$. The \ref{alg:GS} algorithm projects the current iterate $y^k$ on the image of $A$ using the orthogonal projector $A A^\dag$ and adjusts to $b_i$ the amplitude of each coordinate. We describe this method explicitly below.

\begin{algorithm}[ht]
\caption{Gerchberg-Saxton.}
\begin{algorithmic} [1]
\REQUIRE An initial $y^1\in\mathbf{F}$, i.e. such that $|y^1|=b$.
\FOR{$k=1,\ldots,N-1$}
\STATE Set
\BEQ\label{alg:GS}\tag{Gerchberg-Saxton}
y_i^{k+1} = b_i\, \frac{(A A^\dag y^k)_i} {|(A A^\dag y^k)_i|},\quad i=1,\ldots,n.
\EEQ
\ENDFOR
\ENSURE $y_N \in \mathbf{F}$.
\end{algorithmic} 
\end{algorithm} 

Because $\bf F$ is not convex however, this alternating projection method usually converges to a stationary point $y^\infty$ which does not belong to the intersection of $\bf F$ with the image of $A$, and hence $|A A^\dag y^\infty| \neq b$. Several modifications proposed in \citep{Fien82} improve the convergence rate but do not eliminate the existence of multiple stationary points. To guarantee convergence to a unique solution,
which hopefully belongs to the intersection of $\bf F$ and the image of $A$, this non-convex optimization problem has recently been relaxed as a semidefinite program \citep{Chai11,Cand11}, where phase recovery is formulated as a matrix completion problem (described in Section \ref{s:completion}). Although the computational complexity of this relaxation is much higher than that of the~\ref{alg:GS} algorithm, it is able to recover $x$ from $|A x|$ (up to a multiplicative constant) in a number of cases \citep{Chai11,Cand11}.

\subsection{Splitting Phase and Amplitude Variables}
As opposed to these strategies, we solve the phase recovery problem by explicitly separating the amplitude and phase variables, and by only optimizing the values of the phase variables. In the noiseless case, we can write $Ax = \diag(b)u$ where $u \in \complex^n$ is a phase vector, satisfying $|u_i| = 1$ for $i=1,\ldots,n$. Given $b = |A x|$, the phase recovery problem can thus be written as
\[
\min_{\substack{u\in\complex^n,\,|u_i|=1,\\x\in\complex^p}} ~\|Ax-\diag(b)u\|_2^2,
\]
where we optimize over both variables $u\in\complex^n$ and $x\in\complex^p$. In this format, the inner minimization problem in $x$ is a standard least squares and can be solved explicitly by setting
\[
x=A^\dag\diag(b)u,
\]
which means that problem~\eqref{eq:ph-recov} is equivalent to the reduced problem
$$
\min_{\substack{|u_i|=1\\u\in\complex^n}} ~ \|AA^\dag\diag(b)u-\diag(b)u\|_2^2.
$$
The objective of this last problem can be rewritten as follows
\BEAS
\|AA^\dag\diag(b)u-\diag(b)u\|_2^2 &=& \|(AA^\dag-\idm)\diag(b)u\|_2^2\\
&=& u^*\diag(b^T)\tilde M\diag(b)u.
\EEAS
where $\tilde M=(AA^\dag-\idm)^*(AA^\dag-\idm)=\idm-AA^\dag$. Finally, the phase recovery problem~\eqref{eq:ph-recov} becomes
\BEQ\label{eq:ph-partit}
\BA{ll}
\mbox{minimize} & u^*Mu\\
\mbox{subject to} & |u_i|=1,\quad i=1,\ldots n,
\EA\EEQ
in the variable $u\in\complex^n$, where the Hermitian matrix
\[
M=\diag(b)(\idm-AA^\dag)\diag(b)
\] 
is positive semidefinite. The intuition behind this last formulation is simple, $(\idm-AA^\dag)$ is the orthogonal projector on the orthogonal complement of the image of~$A$ (the kernel of $A^*$), so this last problem simply minimizes in the phase vector $u$ the norm of the component of $\diag(b)u$ which is not in the image of~$A$.

\subsection{Greedy Optimization in Phase}\label{s:greedy}

Having transformed the phase recovery problem~\eqref{eq:ph-recov} in the quadratic minimization problem~\eqref{eq:ph-partit}, suppose that we are given an initial vector $u\in\complex^n$, and focus on optimizing over a single component $u_i$ for $i=1,\ldots,n$. The problem is equivalent to solving
$$
\BA{ll}
\mbox{minimize} & \bar u_i M_{ii}u_i + 2\Re\left(\sum_{j\neq i}  \bar u_j M_{ji} u_i\right)\\
\mbox{subject to} & |u_i|=1,\quad i=1,\ldots n,
\EA$$
in the variable $u_i\in\complex$ where all the other phase coefficients $u_j$ remain constant. Because $|u_i|=1$ this then amounts to solving
$$
\min_{|u_i|=1}~\Re\left(u_i\sum_{j\neq i}  M_{ji} \bar u_j\right)
$$
which means
\BEQ \label{eq:greedy-coef}
u_i=\frac{-\sum_{j\neq i} M_{ji} \bar u_j}{\left| \sum_{j\neq i} M_{ji} \bar u_j \right|}
\EEQ
for each $i=1,\ldots,n$, when $u$ is the optimum solution to problem~\eqref{eq:ph-partit}. We can use this fact to derive Algorithm~\ref{alg:greedy}, a greedy algorithm for optimizing the phase problem.

\begin{algorithm}[ht]  
\caption{Greedy algorithm in phase.\label{alg:greedy}} 
\begin{algorithmic} [1]
\REQUIRE An initial $u\in\complex^n$ such that $|u_i|=1$, $i=1,\ldots,n$. An integer $N>1$.
\FOR{$k=1,\ldots,N$}
\FOR {$i = 1,\ldots n$}
\STATE Set
$$
u_i=\frac{-\sum_{j\neq i} M_{ji} \bar u_j}{\left| \sum_{j\neq i} M_{ji} \bar u_j\right|}
$$
\ENDFOR
\ENDFOR
\ENSURE $u\in\complex^n$ such that $|u_i|=1$, $i=1,\ldots,n$.
\end{algorithmic} 
\end{algorithm} 

This greedy algorithm converges to a stationary point $u^\infty$, but it is generally not a global solution of problem~\eqref{eq:ph-partit}, and hence $|A A^\dag \diag(u^\infty) b| \neq b$. It has often nearly the same stationary points as the \ref{alg:GS} algorithm. One can indeed verify that if $u^\infty$ is a stationary point then $y^\infty = \diag(u^\infty) b$ is a stationary point of the \ref{alg:GS} algorithm. Conversely if $b$ has no zero coordinate and $y^\infty$ is a stable stationary point of the \ref{alg:GS} algorithm then $u_i^\infty = y_i^\infty / |y_i^\infty|$ defines a stationary point of the greedy algorithm in phase.

If $A x$ can be computed with a fast algorithm using $O( n \log n)$ operations, which is the case for Fourier or wavelets transform operators for example, then each~\ref{alg:GS} iteration is computed with $O(n \log n)$ operations. The greedy phase algorithm above does not take advantage of this fast algorithm and requires $O(n^2)$ operations to update all coordinates $u_i$ for each iteration $k$. However, we will see in Section~\ref{s:struct} that a small modification of the algorithm allows for $O( n \log n)$ iteration complexity.

\subsection{Complex {{\em MaxCut}}}
Following the classical relaxation argument in \citep{Shor87,Lova91,goem95,Nest98b}, we first write $U=uu^*\in\herm_n$. Problem~\eqref{eq:ph-partit}, written
\[\BA{rll}
QP(M)  \triangleq &\mbox{min.} & u^*Mu\\
& \mbox{subject to} & |u_i|=1,\quad i=1,\ldots n,
\EA\]
in the variable $u\in\complex^n$, is equivalent to 
\[
\BA{ll}
\mbox{min.} & \Tr(UM)\\
\mbox{subject to} & \diag(U)=1\\
& U\succeq 0,\,\Rank(U)=1,
\EA\]
in the variable $U\in\herm_n$. After dropping the (nonconvex) rank constraint, we obtain the following convex relaxation
\BEQ\label{eq:ph-SDP}\tag{PhaseCut}
\BA{rll}
SDP(M)  \triangleq & \mbox{min.} & \Tr(UM)\\
& \mbox{subject to} & \diag(U)=1,\,U\succeq 0,
\EA\EEQ
which is a semidefinite program (SDP) in the matrix $U\in\herm_n$ and can be solved efficiently. When the solution of problem~\ref{eq:ph-SDP} has rank one, the relaxation is tight and the vector $u$ such that $U=uu^*$ is an optimal solution of the phase recovery problem~\eqref{eq:ph-partit}. If the solution has rank larger than one, a normalized leading eigenvector $v$ of $U$ is used as an approximate solution, and $\diag(U-vv^T)$ gives a measure of the uncertainty around the coefficients of $v$.

In practice, semidefinite programming solvers are rarely designed to directly handle problems written over Hermitian matrices and start by reformulating complex programs in $\herm_n$ as real semidefinite programs over~$\symm_{2n}$ based on the simple facts that follow. For $Z,Y\in\herm_n$, we define ${\mathcal T}(Z)\in\symm_{2n}$ as in \citep{Goem04}
\BEQ\label{eq:T-op}
{\mathcal T}(Z)=
\left(\BA{cc}
\Re(Z) & -\Im(Z)\\
\Im(Z) & \Re(Z)
\EA\right)
\EEQ
so that $\Tr({\mathcal T}(Z){\mathcal T}(Y))=2\Tr(ZY)$. By construction, $Z\in\herm_n$ iff ${\mathcal T}(Z)\in\symm_{2n}$. One can also check that $z=x+iy$ is an eigenvector of $Z$ with eigenvalue $\lambda$ if and only if
\[
\left(\BA{c}
x\\
y
\EA\right)
~~\mbox{and}~~
\left(\BA{c}
-y\\
x
\EA\right)
\]
are eigenvectors of ${\mathcal T}(Z)$, both with eigenvalue $\lambda$ (depending on the normalization of $z$, one corresponds to $(\Re(z),\Im(z))$, the other one to $(\Re(i\,z),\Im(i\,z))$. This means in particular that $Z\succeq 0$ if and only if ${\mathcal T}(Z)\succeq 0$. 

We can use these facts to formulate an equivalent semidefinite program over real symmetric matrices, written
\[\BA{ll}
\mbox{minimize} & \Tr({\mathcal T}(M)X)\\
\mbox{subject to} & X_{i,i}+X_{n+i,n+i}=2\\
& X_{i,j} = X_{n+i,n+j},\,X_{n+i,j} = -X_{i,n+j}, \quad i,j=1,\ldots,n,\\
& X\succeq 0,
\EA\]
in the variable $X$ in $\symm_{2n}$. This last problem is equivalent to~\ref{eq:ph-SDP}. In fact, because of symmetries in ${\mathcal T}(M)$, the equality constraints enforcing symmetry can be dropped, and this problem is equivalent to a {\em MaxCut} like problem in dimension $2n$, which reads
\BEQ\label{eq:sdp-tau}
\BA{ll}
\mbox{minimize} & \Tr({\mathcal T}(M)X)\\
\mbox{subject to} & \diag(X)=1, X\succeq 0,
\EA\EEQ
in the variable $X$ in $\symm_{2n}$. As we will see below, formulating a relaxation to the phase recovery problem as a complex {\em MaxCut}-like semidefinite program has direct computational benefits. 

\vskip 6ex
\section{Algorithms}\label{s:algos}
In the previous section, we have approximated the phase recovery problem~\eqref{eq:ph-partit} by a convex relaxation, written
\[
\BA{ll}
\mbox{minimize} & \Tr(UM)\\
\mbox{subject to} & \diag(U)=1,\,U\succeq 0,
\EA\]
which is a semidefinite program in the matrix $U\in\herm_n$. The dual, written
\BEQ\label{eq:dual-ph-SDP}
\max_{w\in\reals^n}~n \lambda_{\mathrm{min}}(M+\diag(w))-1^Tw,
\EEQ
is a minimum eigenvalue maximization problem in the variable $w\in\reals^n$. Both primal and dual can be solved efficiently. When exact phase recovery is possible, the optimum value of the primal problem~\ref{eq:ph-SDP} is zero and we must have $\lambda_{\mathrm{min}}(M)=0$, which means that $w=0$ is an optimal solution of the dual.

\subsection{Interior Point Methods}
For small scale problems, with $n\sim10^2$, generic interior point solvers such as SDPT3 \citep{Toh96} solve problem~\eqref{eq:sdp-tau} with a complexity typically growing as $O\left(n^{4.5}\log(1/\epsilon)\right)$ where $\epsilon>0$ is the target precision \citep[\S4.6.3]{Bent01}. Exploiting the fact that the $2n$ equality constraints on the diagonal in~\eqref{eq:sdp-tau} are singletons, \citet{Helm96} derive an interior point method for solving the {\em MaxCut} problem, with complexity growing as $O\left(n^{3.5}\log(1/\epsilon)\right)$ where the most expensive operation at each iteration is the inversion of a positive definite matrix, which costs $O(n^3)$ flops.

\subsection{First-Order Methods}
When $n$ becomes large, the cost of running even one iteration of an interior point solver rapidly becomes prohibitive. However, we can exploit the fact that the dual of problem~\eqref{eq:sdp-tau} can be written (after switching signs) as a maximum eigenvalue minimization problem. Smooth first-order minimization algorithms detailed in \citep{Nest04a} then produce an $\epsilon$-solution after
\[
O\left(\frac{n^3\sqrt{\log n}}{\epsilon}\right)
\]
floating point operations. Each iteration requires forming a matrix exponential, which costs $O(n^3)$ flops. This is not strictly smaller than the iteration complexity of specialized interior point algorithms, but matrix structure often allows significant speedup in this step. Finally, the simplest subgradient methods produce an $\epsilon$-solution in
\[
O\left(\frac{n^2\log n}{\epsilon^2}\right)
\]
floating point operations. Each iteration requires computing a leading eigenvector which has complexity roughly $O(n^2\log n)$.

\subsection{Block Coordinate Descent}
We can also solve the semidefinite program in~\ref{eq:ph-SDP} using a block coordinate descent algorithm. While no explicit complexity bounds are available for this method in our case, the algorithm is particularly simple and has a very low cost per iteration (it only requires computing a matrix vector product). We write $i^c$ the index set $\{1,\ldots,i-1,i+1,\ldots,n\}$ and describe the method as Algorithm~\ref{alg:block}.

Block coordinate descent is widely used to solve statistical problems where the objective is separable (LASSO is a typical example) and was shown to  efficiently solve semidefinite programs arising in covariance estimation \citep{dAsp06b}. These results were extended by \citep{Wen12} to a broader class of semidefinite programs, including {\em MaxCut}. We briefly recall its simple construction below, applied to a barrier version of the {\em MaxCut} relaxation~\ref{eq:ph-SDP}, written
\BEQ \label{eq:bar-sdp}
\BA{ll}
\mbox{minimize} & \Tr(UM)- \mu \log\det(U) \\
\mbox{subject to} & \diag(U)=1
\EA\EEQ
which is a semidefinite program in the matrix $U\in\herm_n$, where $\mu>0$ is the barrier parameter. As in interior point algorithms, the barrier enforces positive semidefiniteness and the value of $\mu>0$ precisely controls the distance between the optimal solution to~\eqref{eq:bar-sdp} and the optimal set of~\ref{eq:ph-SDP}. We refer the reader to \citep{Boyd03} for further details. The key to applying coordinate descent methods to problems penalized by the $\log\det(\cdot)$ barrier is the following block-determinant formula
\BEQ\label{eq:block-det}
\det(U)=\det(B) \, \det(y-x^TB^{-1}x),
\quad \mbox{when}\quad
U=\left(\BA{cc}
B & x\\
x^T & y\\
\EA\right), \quad U \succ 0.
\EEQ
This means that, all other parameters being fixed, minimizing the function $\det(X)$ in the row and column block of variables $x$, is equivalent to minimizing the quadratic form $y-x^TZ^{-1}x$, arguably a much simpler problem. Solving the semidefinite program~\eqref{eq:bar-sdp} row/column by row/column thus amounts to solving the simple problem~\eqref{eq:block-pb} described in the following lemma.

\begin{lemma}\label{lem:block}
Suppose  $\sigma>0$, $c\in\reals^{n-1}$, and $B\in\symm_{n-1}$ are such that $b\neq 0$ and $B\succ 0$, then the optimal solution of the block problem
\BEQ\label{eq:block-pb}
\min_x ~ c^Tx - \sigma \log(1-x^TB^{-1}x)
\EEQ
is given by
\[
x=\frac{\sqrt{\sigma^2+\gamma}-\sigma}{\gamma}Bc
\]
where $\gamma=c^TBc$.
\end{lemma}
\begin{proof}
As in \citep{Wen12}, a direct consequence of the first order optimality conditions for~\eqref{eq:block-pb}.
\end{proof}

Here, we see problem~\eqref{eq:bar-sdp} as an unconstrained minimization problem over the off-diagonal coefficients of~$U$, and~\eqref{eq:block-det} shows that each block iteration amounts to solving a minimization subproblem of the form~\eqref{eq:block-pb}. Lemma~\ref{lem:block} then shows that this is equivalent to computing a matrix vector product. 
Linear convergence of the algorithm is guaranteed by the result in~\citep[\S9.4.3]{Boyd03} and the fact that the function $\log\det$ is strongly convex over compact subsets of the positive semidefinite cone. So the complexity of the method is bounded by $O\left(\log\frac{1}{\epsilon}\right)$ but the constant in this bound depends on $n$ here, and the dependence cannot be quantified explicitly.

\begin{algorithm}[ht]  
\caption{Block Coordinate Descent Algorithm for {\bf PhaseCut}.\label{alg:block}} 
\begin{algorithmic} [1]
\REQUIRE An initial $X^0=\idm_n$ and $\nu>0$ (typically small). An integer $N>1$.
\FOR{$k=1,\ldots,N$}
\STATE Pick $i\in[1,n]$.
\STATE Compute
\[
x=X^k_{i^c,i^c}M_{i^c,i}
\quad \mbox{and}\quad
\gamma=x^* M_{i^c,i}
\]
\STATE If $\gamma>0$, set
\[
X^{k+1}_{i^c,i}=X^{k+1 *}_{i,i^c}=-\sqrt{\frac{1-\nu}{\gamma}}x
\]
else
\[
X^{k+1}_{i^c,i}=X^{k+1 *}_{i,i^c}=0.
\]
\ENDFOR
\ENSURE A matrix $X\succeq 0$ with $\diag(X)=1$.
\end{algorithmic} 
\end{algorithm}

\subsection{Initialization \& Randomization}
Suppose the Hermitian matrix $U$ solves the semidefinite relaxation~\ref{eq:ph-SDP}. As in \citep{Goem04,Ben-03,Zhan06,So07}, we generate complex Gaussian vectors $x\in\complex^n$ with $x\sim\mathcal{N}_\complex(0,U)$, and for each sample $x$, we form $z\in\complex^n$ such that
\[
z_i=\frac{x_i}{|x_i|},\quad i=1,\ldots,n.
\]
All the sample points $z$ generated using this procedure satisfy $|z_i|=1$, hence are feasible points for problem~\eqref{eq:ph-partit}. This means in particular that $QP(M)\leq\Expect[z^*Mz]$. In fact, this expectation can be computed almost explicitly, using
\[
\Expect[zz^*]=F(U),\quad \mbox{with} \quad F(w)=\frac{1}{2}e^{i\,\mathrm{arg}(w)} \displaystyle \int_0^\pi \cos(\theta) \arcsin(|w| \cos(\theta)) d\theta
\]
where $F(U)$ is the matrix with coefficients $F(U_{ij})$, $i,j=1,\ldots,n$. We then get
\BEQ\label{eq:approx-e}
SDP(M) \leq QP(M) \leq \Tr(MF(U))
\EEQ
In practice, to extract good candidate solutions from the solution $U$ to the SDP relaxation in~\ref{eq:ph-SDP}, we sample a few points from $\mathcal{N}_\complex(0,U)$, normalize their coordinates and simply pick the point which minimizes~$z^*Mz$.

This sampling procedure also suggests a simple spectral technique for computing rough solutions to problem~\ref{eq:ph-SDP}: compute an eigenvector of $M$ corresponding to its lowest eigenvalue and simply normalize its coordinates (this corresponds to the simple bound on {\em MaxCut} by \citep{Delo93}). The information contained in $U$ can also be used to solve a robust formulation \citep{Ben09} of problem~\eqref{eq:ph-recov} given a Gaussian model $u\sim\mathcal{N}_\complex(0,U)$.

\subsection{Approximation Bounds} \label{ss:approx}
% AA: add sentence saying that optimization is in the wrong direction, so bounds are not as interesting as in the MaxCut case.
The semidefinite program in~\ref{eq:ph-SDP} is a {\em MaxCut}-type graph partitioning relaxation whose performance has been studied extensively. Note however that most approximation results for {\em MaxCut} study {\em maximization} problems over positive semidefinite or nonnegative matrices, while we are {\em minimizing} in~\ref{eq:ph-SDP} so, as pointed out in \citep{Kisi10,So10a} for example, we do not inherit the constant approximation ratios that hold in the classical {\em MaxCut} setting. 

%The approximation results in \citep*{Goem04,Ben-03,Zhan06,So07} use the randomization argument above to show that the optimum SDP(W) of~\ref{eq:ph-SDP} approximates that of problem~\eqref{eq:ph-recov} with an approximation ratio of $\pi/4$ when the objective matrix $W\in\herm_n$ is negative semidefinite (all the results cited above are focused on maximization problems, hence the signs are switched), i.e.
%\[
%SDP(-W) \leq QP(-W) \leq \frac{\pi}{4} SDP(-W).
%\]
%A similar bound (in $\pi/2$) holds in the binary case where $u\in\reals^n$, and this last bound cannot be improved, as shown in~\citep{Alon04}. If we only assume that $\diag(W)=0$, then the references above also show that the optimal values of problems~\eqref{eq:ph-partit} and~\ref{eq:ph-SDP} satisfy $QP(W)<0$ and
%\[
%SDP(W) \leq QP(W) \leq \frac{c}{\log n} SDP(W),
%\]
%where $c>0$ is an absolute constant, \citep{Ben-03} show that this bound too is unimprovable without further assumptions on the structure of $A$. In our case, setting $W=M-\lambdamax(M)\idm$  means that
%\[
%SDP(M) \leq QP(M) \leq \frac{\pi}{4} SDP(M) + \left(1-\frac{\pi}{4}\right) n \lambdamax(M).
%\]
%% AA: DEBUG: check position of parenthesis here....
%This produces a bound on the quality of the approximation of the phase recovery problem~\eqref{eq:ph-partit} by the semidefinite relaxation~\ref{eq:ph-SDP}. More importantly, we will see in the next section that we can also obtain explicit conditions for this relaxation to be exact.

\subsection{Exploiting Structure}\label{s:struct}
In some instances, we have additional structural information on the solution of problems~\eqref{eq:ph-recov} and~\eqref{eq:ph-partit}, which usually reduces the complexity of approximating~\ref{eq:ph-SDP} and improves the quality of the approximate solutions. We briefly highlight a few examples below.

\subsubsection{Symmetries}
In some cases, e.g. signal processing examples where the signal is symmetric, the optimal solution $u$ has a known symmetry pattern. For example, we might have $u(k_- - i)=u(k_+ +i)$ for some $k_-,k_+$ and indices $i\in[0,k_--1]$. This means that the solution $u$ to problem~\eqref{eq:ph-recov} can be written $u=Pv$, where $v\in\complex^q$ with $q<n$, and we can solve~\eqref{eq:ph-recov} by focusing on the smaller problem
\[\BA{ll}
\mbox{minimize} & v^*P^*MPv\\
\mbox{subject to} & |(Pv)_i|=1,\quad i=1,\ldots n,
\EA\]
in the variable $v\in\complex^q$. We reconstruct a solution $u$ to~\eqref{eq:ph-recov} from a solution $v$ to the above problem as $u=Pv$. This produces significant computational savings.

\subsubsection{Alignment}
In other instances, we might have prior knowledge that the phases of certain samples are aligned, i.e. that there is an index set $I$ such that $u_i=u_j, \quad \mbox{for all } i,j\in I$, this reduces to the symmetric case discussed above when the phase is arbitrary. W.l.o.g., we can also fix the phase to be one, with $u_i=1$ for $i\in I$, and solve a constrained version of the relaxation~\ref{eq:ph-SDP}
\[\BA{ll}
\mbox{min.} & \Tr(UM)\\
\mbox{subject to} & U_{ij}=1,\quad i,j\in I,\\
& \diag(U)=1,\,U\succeq 0,
\EA\]
which is a semidefinite program in $U\in\herm_n$.

\subsubsection{Fast {F}ourier transform}
If the product $Mx$ can be computed with a fast algorithm in $O( n \log n)$ operations, which is the case for Fourier or wavelet transform operators, we significantly speed up the iterations of Algorithm~\ref{alg:block} to update all coefficients at once. Each iteration of the modified Algorithm~\ref{alg:block} then has cost~$O( n \log n)$ instead of~$O(n^2)$.

\subsubsection{Real valued signal\label{sss:real_valued_signal}}
In some cases, we know that the solution vector $x$ in \eqref{eq:ph-recov} is real valued. Problem~\eqref{eq:ph-recov} can be reformulated to explicitly constrain the solution to be real, by writing it
\[
\min_{\substack{u\in\complex^n,\,|u_i|=1,\\x\in\reals^p}} ~\|Ax-\diag(b)u\|_2^2
\]
or again, using the operator ${\mathcal T}(\cdot)$ defined in~\eqref{eq:T-op}
\[\BA{ll}
\mbox{minimize} & \left\| {\mathcal T} (A)
\left(\BA{c} x \\ 0 \EA\right) 
-
\diag\left(\BA{c} b \\ b \EA\right)
\left(\BA{c}\Re(u) \\ \Im(u)\EA\right)
\right\|_2^2\\
\mbox{subject to} & u\in\complex^n,\,|u_i|=1\\
& x\in\reals^p.
\EA\]
The optimal solution of the inner minimization problem in $x$ is given by $x=A_2^\dag B_2v$, where
\[
A_2=\left(\BA{c} \Re(A) \\ \Im(A) \EA\right),
\quad
B_2=\diag\left(\BA{c} b \\ b \EA\right),
\quad \mbox{and} \quad
v = \left(\BA{c}\Re(u) \\ \Im(u)\EA\right)
\] 
hence the problem is finally rewritten
\[\BA{ll}
\mbox{minimize} & \| (A_2 A_2^\dag B_2- B_2)v\|_2^2\\
\mbox{subject to} & v_i^2+v_{n+i}^2=1,\quad i=1,\ldots,n,
\EA\]
in the variable $v\in\reals^{2n}$. This can be relaxed as above by the following problem
\[\BA{ll}
\mbox{minimize} & \Tr(VM_2)\\
\mbox{subject to} & V_{ii}+V_{n+i,n+i}=1,\quad i=1,\ldots,n,\\
& V \succeq 0,
\EA\]
which is a semidefinite program in the variable $V\in\symm_{2n}$, where $M_2=(A_2 A_2^\dag B_2- B_2)^T(A_2 A_2^\dag B_2- B_2)=B_2^T(\idm-A_2A_2^\dag)B_2$.

\vskip 6ex
\section{Matrix completion \& exact recovery conditions}\label{s:completion}
In \citep{Chai11,Cand11}, phase recovery~\eqref{eq:ph-recov} 
is cast as a matrix completion problem. We briefly review this approach and compare it with the semidefinite program in \ref{eq:ph-SDP}. Given a signal vector $b\in\reals^n$ and a sampling matrix $A\in\complex^{n \times p}$, we look for a vector $x\in\complex^p$ satisfying
\[
|a_i^*x|=b_i, \quad i=1,\ldots,n,
\]
where the vector $a_i^*$ is the $i^{th}$ row of $A$ and $x\in\complex^p$ is the signal we are trying to reconstruct. The phase recovery problem is then written as
\[\BA{ll}
\mbox{minimize} & \Rank(X)\\
\mbox{subject to} & \Tr(a_ia_i^*X)=b_i^2,\quad i=1,\ldots,n\\
& X \succeq 0
\EA\]
in the variable $X\in\herm_p$, where $X=xx^*$ when exact recovery occurs. This last problem can be relaxed as
\BEQ\label{eq:ph-lift}
\BA{ll}
\mbox{minimize} & \Tr(X)\\
\tag{PhaseLift}
\mbox{subject to} & \Tr(a_ia_i^*X)=b_i^2,\quad i=1,\ldots,n\\
& X \succeq 0
\EA\EEQ
which is a semidefinite program (called \ref{eq:ph-lift} by \citet{Cand11}) in the variable $X\in\herm_p$. Recent results in \citep{Cand11a,Cand12} give explicit (if somewhat stringent) conditions on $A$ and $x$ under which the relaxation is tight (i.e. the optimal $X$ in~\ref{eq:ph-lift} is unique, has rank one, with leading eigenvector $x$).

\subsection{Weak Formulation\label{ss:weak_version}}
We also introduce a weak version of \ref{eq:ph-lift}, which is more directly related to \ref{eq:ph-SDP} and is easier to interpret geometrically. It was noted in \citep{Cand11a} that, when $\idm\in\mbox{span}\{a_ia_i^*\}_{i=1}^n$, the condition $\Tr(a_ia_i^*X)=b_i^2,i=1,...,n$ determines $\Tr(X)$, so in this case the trace minimization objective is redundant and \ref{eq:ph-lift} is equivalent to
\BEQ\label{eq:weak-ph-lift}
\BA{ll}
\mbox{find} & X\\
\tag{Weak PhaseLift}
\mbox{subject to} & \Tr(a_ia_i^*X)=b_i^2,\quad i=1,\ldots,n\\
& X \succeq 0.
\EA\EEQ
When $\idm\notin\mbox{span}\{a_ia_i^*\}_{i=1}^n$ on the other hand, \ref{eq:weak-ph-lift} and \ref{eq:ph-lift} are not equivalent: solutions of \ref{eq:ph-lift} solve \ref{eq:weak-ph-lift} too but the converse is not true. Interior point solvers typically pick a solution at the analytic center of the feasible set of~\ref{eq:weak-ph-lift} which in general can be significantly different from the minimum trace solution.

However, in practice, the removal of trace minimization does not really seem to alter the performances of the algorithm. We will illustrate this affirmation with numerical experiments in~\S\ref{ss:role_of_trace} and a formal proof is given in \citep{Dema12} who showed that, in the case of Gaussian random measurements, the relaxation of \ref{eq:weak-ph-lift} was tight with high probability under the same conditions as \ref{eq:ph-lift}.

\subsection{Phase Recovery as a Projection}\label{ss:geometric_point_of_view}
We will see in what follows that phase recovery can interpreted as a projection problem. These results will prove useful later to study stability. The \ref{eq:ph-SDP} reconstruction problem defined in~\ref{eq:ph-SDP} is written
\[\BA{ll}
\mbox{minimize} & \Tr(UM)\\
\mbox{subject to} & \diag(U)=1,\,U\succeq 0,
\EA\]
with $M=\diag(b)(\idm-AA^\dag)\diag(b)$. In what follows, we assume $b_i\ne 0$, $i=1,...,n$, which means that, after scaling $U$, solving \ref{eq:ph-SDP} is equivalent to solving
\BEQ\label{eq:scaled-pb}
\BA{ll}
\mbox{minimize} & \Tr(V(\idm-AA^\dag))\\
\mbox{subject to} & \diag(V)=b^2, V\succeq 0.
\EA\EEQ
In the following lemma, we show that this last semidefinite program can be understood as a projection problem on a section of the semidefinite cone using the trace (or nuclear) norm. We define
\[
\mathcal{F}=\{V\in\herm_n : x^*Vx=0,\forall x\in \range(A)^\perp\}
\]
which is also $\mathcal{F}=\{V\in\herm_n: (\idm-AA^\dag)V(\idm-AA^\dag)=0\}$, and we now formulate the objective of problem~\eqref{eq:scaled-pb} as a distance.

\begin{lemma}\label{lem:dist_to_F}
For all $V\in\herm_n$ such that $V\succeq 0$,
\begin{equation}
\Tr(V(\idm-AA^\dag))=d_1(V,\mathcal{F})
\end{equation}
where $d_1$ is the distance associated to the trace norm.
\end{lemma}
\begin{proof}
Let $\mathcal{B}_1$ (resp. $\mathcal{B}_2$) be an orthonormal basis of $\Range\,A$ (resp. $(\Range\,A)^\perp$). Let $T$ be the transformation matrix from canonical basis to orthonormal basis $\mathcal{B}_1\cup\mathcal{B}_2$. Then 
\[
\mathcal{F}=\{V\in\herm_n\mbox{ s.t. }T^{-1}VT=\left(\begin{smallmatrix}S_1&S_2\\S_2^*&0\end{smallmatrix}\right), S_1\in\herm_p, S_2\in\mathcal{M}_{p,n-p}\}
\]
As the transformation $X\to T^{-1}XT$ preserves the nuclear norm, for every matrix $V\succeq 0$, if we write
\[
T^{-1}VT=\left(\begin{smallmatrix}V_1&V_2\\V_2^*&V_3\end{smallmatrix}\right)
\]
then the orthogonal projection of $V$ onto $\mathcal{F}$ is
\[
W=T\left(\begin{smallmatrix}V_1&V_2\\V_2^*&0\end{smallmatrix}\right)T^{-1},
\]
so $d_1(V,\mathcal{F})=\|V-W\|_1=\|\left(\begin{smallmatrix}0&0\\0&V_3\end{smallmatrix}\right)\|_1$. As $V\succeq 0$, $\left(\begin{smallmatrix}V_1&V_2\\V_2^*&V_3\end{smallmatrix}\right)\succeq 0$ hence $\left(\begin{smallmatrix}0&0\\0&V_3\end{smallmatrix}\right)\succeq 0$, so $d_1(V,\mathcal{F})=\Tr\left(\begin{smallmatrix}0&0\\0&V_3\end{smallmatrix}\right)$. Because $AA^\dag$ is the orthogonal projection onto $\range(A)$, we have
$T^{-1}(\idm-AA^\dag)T=\left(\begin{smallmatrix}0&0\\0&\idm\end{smallmatrix}\right)$ hence
\[
d_1(V,\mathcal{F})=\Tr\left(\begin{smallmatrix}0&0\\0&V_3\end{smallmatrix}\right)=\Tr((T^{-1}VT)(T^{-1}(\idm-AA^\dag)T))=\Tr(V(\idm-AA^\dag))
\]
which is the desired result.
\end{proof}

\begin{figure}
\includegraphics[width=.33\textwidth]{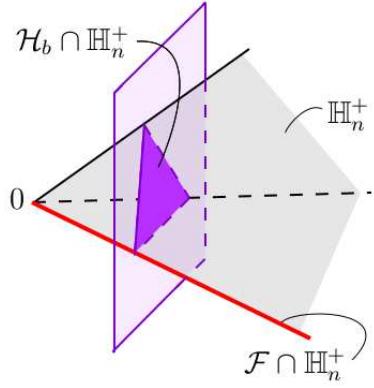} 
\caption{Schematic representation of the sets involved in equations \eqref{eq:ph-cut-geom} and \eqref{eq:ph-lift-geom} : the cone of positive hermitian matrices $\mathbb{H}_n^+$ (in light grey), its intersection with the affine subspace $\mathcal{H}_b$, and $\mathcal{F}\cap\mathbb{H}_n^+$, which is a face of $\mathbb{H}_n^+$.}
\label{geometrical_interpretation}
\end{figure}

This means that \ref{eq:ph-SDP} can be written as a projection problem, i.e.
\BEQ\label{eq:ph-cut-geom}
\BA{ll}
\mbox{minimize} & d_1(V,\mathcal{F})\\
\mbox{subject to} & V\in\herm_n^+\cap\mathcal{H}_b
\EA\EEQ
in the variable $V\in\herm_n$, where $\mathcal{H}_b=\{V\in\herm_n\mbox{ s.t. }V_{i,i}=b_i^2, i=1,...,n\}$. Moreover, with $a_i$ the $i$-th row of $A$, we have for all $X\in\herm_p^+$, $\Tr(a_ia_i^*X)=a_i^*Xa_i=\diag(AXA^*)_i, \quad i=1,\ldots,n$, so if we call $V=AXA^*\in\mathcal{F}$, when $A$ is injective, $X=A^\dag VA^{\dag *}$ and \ref{eq:weak-ph-lift} is equivalent to
\[
\BA{ll}
\mbox{find} & V\in\herm_n^+\cap\mathcal{F}\\
\mbox{subject to} & \diag(V)=b^2.
\EA\]
First order algorithms for \ref{eq:weak-ph-lift} will typically solve
\[
\BA{ll}
\mbox{minimize} & d(\diag(V),b^2)\\
\mbox{subject to} & V\in\herm_n^+\cap\mathcal{F}
\EA\]
for some distance $d$ over $\R^n$. If $d$ is the $l^s$-norm, for any $s\geq 1$, $d(\diag(V),b^2)=d_s(V,\mathcal{H}_b)$, where $d_s$ is the distance generated by the Schatten $s$-norm, the algorithm becomes
\BEQ\label{eq:ph-lift-geom}
\BA{ll}
\mbox{minimize} & d_s(V,\mathcal{H}_b)\\
\mbox{subject to} & V\in\herm_n^+\cap\mathcal{F}
\EA\EEQ
which is another projection problem in $V$.

Thus, \ref{eq:ph-SDP} and \ref{eq:weak-ph-lift} are comparable, in the sense that both algorithms aim at finding a point of $\mathbb{H}_n^+\cap\mathcal{F}\cap\mathcal{H}_b$ but \ref{eq:ph-SDP} does so by picking a point of $\mathbb{H}_n^+\cap\mathcal{H}_b$ and moving towards $\mathcal{F}$ while \ref{eq:weak-ph-lift} moves a point of $\mathbb{H}_n^+\cap\mathcal{F}$ towards $\mathcal{H}_b$. We can push the parallel between both relaxations much further. We will show in what follows that, in a very general case,~\ref{eq:ph-lift} and a modified version of~\ref{eq:ph-SDP} are simultaneously tight. We will also be able to compare the stability of \ref{eq:weak-ph-lift} and \ref{eq:ph-SDP} when measurements become noisy.

\subsection{Tightness of the Semidefinite Relaxation\label{ss:relaxation_tightness}}
We will now formulate a refinement of the semidefinite relaxation in~\ref{eq:ph-SDP} and prove that this refinement is equivalent to the relaxation in~\ref{eq:ph-lift} under mild technical assumptions. Suppose $u$ is the optimal phase vector, we know that the optimal solution to~\eqref{eq:ph-recov} can then be written $x=A^\dag \diag(b) u$, which corresponds to the matrix $X=A^\dag \diag(b) uu^* \diag(b) A^{\dag *}$ in~\ref{eq:ph-lift}, hence
\[
\Tr(X)=\Tr(\diag(b) A^{\dag *}A^\dag \diag(b) uu^*).
\]
Writing $B=\diag(b) A^{\dag *}A^\dag \diag(b)$, when problem~\eqref{eq:ph-recov} is solvable, we look for the ``minimum trace'' solution among all the optimal points of relaxation~\ref{eq:ph-SDP} by solving
\BEQ\label{eq:ph-SDP2}\tag{PhaseCutMod}
\BA{rll}
SDP2(M) \triangleq & \mbox{min.} & \Tr(BU)\\
& \mbox{subject to} & \Tr(MU)=0\\
& & \diag(U)=1,\, U \succeq 0,
\EA\EEQ
which is a semidefinite program in $U\in\herm_n$. When problem~\eqref{eq:ph-recov} is solvable, then every optimal solution of the semidefinite relaxation~\ref{eq:ph-SDP} is a feasible point of relaxation~\ref{eq:ph-SDP2}. In practice, the semidefinite program $SDP(M+\gamma B)$, written
\[
\BA{ll}
\mbox{minimize} & \Tr((M+\gamma B)U)\\
\mbox{subject to} & \diag(U)=1,\, U \succeq 0,
\EA\]
obtained by replacing $M$ by $M+\gamma B$ in problem~\ref{eq:ph-SDP}, will produce a solution to~\ref{eq:ph-SDP2} whenever $\gamma>0$ is sufficiently small (this is essentially the exact penalty method detailed in \cite[\S4.3]{Bert98} for example). This means that all algorithms (greedy or SDP) designed to solve the original~\ref{eq:ph-SDP} problem can be recycled to solve~\ref{eq:ph-SDP2} with negligible effect on complexity. We now show that the~\ref{eq:ph-SDP2} and~\ref{eq:ph-lift} relaxations are simultaneously tight when $A$ is injective. An earlier version of this paper showed \ref{eq:ph-lift} tightness implies \ref{eq:ph-SDP} tightness and the argument was reversed in \citep{Voro12a} under mild additional assumptions.

\begin{proposition}\label{prop:Phi}
Assume that $b_i \ne 0$ for $i=1,\ldots,n$, that $A$ is injective and that there is a solution $x$ to~\eqref{eq:ph-recov}. The function 
\BEAS
\Phi:\herm_p & \rightarrow & \herm_n\\
X & \mapsto & \Phi(X)=\diag(b)^{-1}AXA^*\diag(b)^{-1}
\EEAS
is a bijection between the feasible points of~\ref{eq:ph-SDP2} and those of~\ref{eq:ph-lift}.
\end{proposition}
\begin{proof}
Note that $\Phi$ is injective whenever $b>0$ and $A$ has full rank. We have to show that $U$ is a feasible point of~\ref{eq:ph-SDP2} if and only if it can be written under the form $\Phi(X)$, where $X$ is feasible for~\ref{eq:ph-lift}.
We first show that
\BEQ\label{eq:U-cond}
\Tr(MU)=0, \quad U\succeq 0,
\EEQ
is equivalent to 
\BEQ\label{eq:Phi-cond}
U=\Phi(X)
\EEQ
for some $X\succeq 0$. Observe that $\Tr(UM)=0$ means $UM=0$ because $U,M\succeq 0$, hence 
$\Tr(MU)=0$ in~\eqref{eq:U-cond} is equivalent to
\[
AA^\dag \diag(b) U \diag(b)=\diag(b)U\diag(b)
\]
because $b>0$ and $M=\diag(b)(\idm-AA^\dag)\diag(b)$. If we set $X=A^\dag \diag(b) U \diag(b)A^{\dag *}$, this last equality implies both % Explicit form for X
\[
AX=AA^\dag \diag(b) U \diag(b)A^{\dag *}=\diag(b) U \diag(b)A^{\dag *}
\] 
and
\[
AXA^*= \diag(b) U \diag(b)A^{\dag *}A^*=\diag(b) U \diag(b)
\]
which is $U=\Phi(X)$, and shows~\eqref{eq:U-cond} implies~\eqref{eq:Phi-cond}. Conversely, if $U=\Phi(X)$ then $\diag(b)U\diag(b)=AXA^*$ and using $AA^\dag A=A$, we get $AXA^*=AA^\dag AXA^*=AA^\dag \diag(b)U\diag(b)$ which means $MU=0$, hence \eqref{eq:U-cond} is in fact equivalent to \eqref{eq:Phi-cond} since $U\succeq 0$ by construction.

Now, if $X$ is feasible for~\ref{eq:ph-lift}, we have shown $\Tr(M\Phi(X))=0$ and $\phi(X)\succeq 0$, moreover $\diag(\Phi(X))_i=\Tr(a_ia_i^*X)/b_i^2=1$, so $U=\Phi(X)$ is a feasible point of~\ref{eq:ph-SDP2}. Conversely, if 
$U$ is feasible for~\ref{eq:ph-SDP2}, we have shown that there exists $X\succeq 0$ such that $U=\Phi(X)$ which means $\diag(b)U\diag(b)=AXA^*$. We also have $\Tr(a_ia_i^*X)=b_i^2U_{ii}=b_i^2$, which means $X$ is feasible for~\ref{eq:ph-lift} and concludes the proof.
\end{proof}

We now have the following central corollary showing the equivalence between~\ref{eq:ph-SDP2} and~\ref{eq:ph-lift} in the noiseless case.

\begin{corollary}\label{cor:equivalence}
If $A$ is injective, $b_i\ne 0$ for all $i=1,...,n$ and if the reconstruction problem~\eqref{eq:ph-recov} admits an exact solution, then~\ref{eq:ph-SDP2} is tight (i.e. has a unique rank one solution) whenever~\ref{eq:ph-lift} is.
\end{corollary}
\begin{proof} 
When $A$ is injective, $\Tr(X)=\Tr(B\Phi(X))$ and $\Rank(X)=\Rank(\Phi(X))$.
\end{proof}

This last result shows that in the noiseless case, the relaxations~\ref{eq:ph-lift} and~\ref{eq:ph-SDP2} are in fact equivalent. In the same way, we could have shown that~\ref{eq:weak-ph-lift} and~\ref{eq:ph-SDP} were equivalent. The performances of both algorithms may not match however when the information on~$b$ is noisy and perfect recovery is not possible.

\begin{remark} Note that Proposition \ref{prop:Phi} and corollary \ref{cor:equivalence} also hold when the initial signal is real and the measurements are complex. In this case, we define the $B$ in~\ref{eq:ph-SDP2} by $B=B_2A_2^{\dag *}A_2^\dag B_2$ (with the notations of paragraph~\ref{sss:real_valued_signal}). We must also replace the definition of $\Phi$ by $\Phi(X)=B_2^{-1}A_2XA_2^*B_2^{-1}$. Furthermore, all steps in the proof of proposition \ref{prop:Phi} are still valid if we replace $M$ by $M_2$, $A$ by $A_2$ and $\diag(b)$ by $B_2$. The only difference is that now $\frac{1}{b_i^2}\Tr(a_ia_i^*X)=\diag(\Phi(X))_{i}+\diag(\Phi(X))_{n+i}$.
\end{remark}

\subsection{Stability in the Presence of Noise.\label{ss:noisy_measurements}}

We now consider the case where the vector of measurements $b$ is of the form $b=|Ax_0|+b_{\rm noise}$. We first introduce a definition of $C$-stability for~\ref{eq:ph-SDP} and~\ref{eq:weak-ph-lift}. The main result of this section is that, when the~\ref{eq:weak-ph-lift} solution in~\eqref{eq:ph-lift-geom} is stable at a point, \ref{eq:ph-SDP} is stable too, with a constant of the same order. The converse does not seem to be true when $b$ is sparse.

\begin{definition}\label{def:stability}
Let $x_0\in\C^n,C>0$. The algorithm \ref{eq:ph-SDP} (resp.~\ref{eq:weak-ph-lift}) is said to be $C$-stable at $x_0$ iff for all $b_{\rm noise}\in\R^n$ close enough to zero, every minimizer $V$ of equation~\eqref{eq:ph-cut-geom} (resp. \eqref{eq:ph-lift-geom}) with $b=|Ax_0|+b_{\rm noise}$, satisfies
\begin{equation*}
\|V-(Ax_0)(Ax_0)^*\|_2\leq C\|Ax_0\|_2\|b_{\rm noise}\|_2.
\end{equation*}
\end{definition}

The following matrix perturbation result motivates this definition, by showing that a $C$-stable algorithm generates a $O(C\|b_{\rm noise}\|_2)$-error over the signal it reconstructs.
\begin{proposition}
Let $C>0$ be arbitrary. We suppose that $Ax_0\ne 0$ and $\|V-(Ax_0)(Ax_0)^*\|_2\leq C\|Ax_0\|_2\|b_{\rm noise}\|_2\leq\|Ax_0\|_2^2/2$. Let $y$ be $V$'s main eigenvector, normalized so that $(Ax_0)^*y=\|Ax_0\|_2$. Then
\begin{equation*}
\|y-Ax_0\|_2=O(C\|b_{\rm noise}\|_2),
\end{equation*}
and the constant in this last equation does not depend upon $A$, $x_0$, $C$ or $\|b\|_2$.
\end{proposition}
\begin{proof}
We use \cite[Eq.10]{El-K09} for
\begin{equation*}
u=\frac{Ax_0}{\|Ax_0\|_2}\hskip 0.8cm
v=\frac{y}{\|Ax_0\|_2}\hskip 0.8cm
E = \frac{V-(Ax_0)(Ax_0)^*}{\|Ax_0\|_2^2}
\end{equation*}
This result is based on \cite[Eq.\,3.29]{Kato95}, which gives a precise asymptotic expansion of $u-v$. For our purposes here, we only need the first-order term. See also \cite{Bhat97}, \cite{Stew90} or \cite{Stew01} among others for a complete discussion. We get $\|v-u\|=O(\|E\|_2)$ because if $M=uu^*$, then $\|R\|_\infty=1$ in \cite[Eq.10]{El-K09}. This implies
\begin{equation*}
\|y-Ax_0\|_2=\|Ax_0\|_2\|u-v\|=O\left(\frac{\|V-(Ax_0)(Ax_0)^*\|_2}{\|Ax_0\|_2}\right)=O(C\|b_{\rm noise}\|)
\end{equation*}
which is the desired result.
\end{proof}

Note that normalizing $y$ differently, we would obtain $\|y-Ax_0\|_2\leq 4C\|b_{\rm noise}\|_2$. We now show the main result of this section, according to which \ref{eq:ph-SDP} is ``almost as stable as'' ~\ref{eq:weak-ph-lift}. In practice of course, the exact values of the stability constants has no importance, what matters is that they are of the same order.

\begin{theorem}\label{thm:stability}
Let $A\in\C^{n\times m}$, for all $x_0\in\C^n,C>0$, if~\ref{eq:weak-ph-lift} is $C$-stable in $x_0$, then \ref{eq:ph-SDP} is $(2C+2\sqrt{2}+1)$-stable in $x_0$.
\end{theorem}
\begin{proof}
Let $x_0\in\C^n,C>0$ be such that~\ref{eq:weak-ph-lift} is $C$-stable in $x_0$. $Ax_0$ is a non-zero vector (because, with our definition, neither~\ref{eq:weak-ph-lift} nor \ref{eq:ph-SDP} may be stable in $x_0$ if $Ax_0=0$ and $A\ne 0$). We set $D=2C+2\sqrt{2}+1$ and suppose by contradiction that \ref{eq:ph-SDP} is not $D$-stable in $x_0$. Let $\epsilon>0$ be arbitrary. Let $b_{\rm n,PC}\in\R^n$ be such that $\|b_{\rm n,PC}\|_2\leq \max(\|Ax_0\|_2,\epsilon/2)$ and such that, for $b=|Ax_0|+b_{\rm n,PC}$, the minimizer $V_{PC}$ of \eqref{eq:ph-cut-geom} verifies
\begin{equation*}
\|V_{PC}-(Ax_0)(Ax_0)^*\|_2>D\|Ax_0\|_2\|b_{\rm n,PC}\|_2
\end{equation*}
Such a $V_{PC}$ must exist or \ref{eq:ph-SDP} would be $D$-stable in $x_0$. We call $V_{PC}^\sslash$ the restriction of $V_{PC}$ to $\Range(A)$ (that is, the matrix such that $x^*(V_{PC}^\sslash) y=x^*(V_{PC})y$ if $x,y\in\Range(A)$ and $x^*(V_{PC}^\sslash) y=0$ if $x\in\Range(A)^\perp$ or $y\in\Range(A)^\perp$) and $V_{PC}^\perp$ the restriction of $V_{PC}$ to $\Range(A)^\perp$. Let us set $b_{\rm n,PL}=\sqrt{V_{PC\,ii}^\sslash}-|Ax_0|_{ii}$ for $i=1,...,n$. As $V_{PC}^\sslash\in\herm_n^+\cap\mathcal{F}$, $V_{PC}^\sslash$ minimizes \eqref{eq:ph-lift-geom} for $b=|Ax_0|+b_{\rm n,PL}$ (because $V_{PC}^\sslash\in\mathcal{H}_b$). Lemmas \ref{lem:ecart_V_PC} and \ref{lem:ecart_b} (proven in the appendix) imply that $\|V_{PC}^\sslash-(Ax_0)(Ax_0)^*\|_2>C\|Ax_0\|_2\|b_{\rm n,PL}\|_2$ and $\|b_{\rm n,PL}\|_2\leq\epsilon$. As $\epsilon$ is arbitrary, ~\ref{eq:weak-ph-lift} is not $C$-stable in $x_0$, which contradicts our hypotheses. Consequently, \ref{eq:ph-SDP} is $(2C+2\sqrt{2}+1)$-stable in $x_0$.
\end{proof}

Theorem~\ref{thm:stability} is still true if we replace $2C+2\sqrt{2}+1$ by any $D>2C+\sqrt{2}$. We only have to replace, in the demonstration, the inequality $\|b_{\rm n,PC}\|_2\leq\|Ax_0\|_2$ by $\|b_{\rm n,PC}\|_2\leq\alpha\|Ax_0\|_2$ with $\alpha={D-(2C+\sqrt{2})}/{(1+\sqrt{2})}$. Also, the demonstration of this theorem is based on the fact that, when $V_{PC}$ solves \eqref{eq:ph-cut-geom}, one can construct some $V_{PL}=V_{PC}^\sslash$ close to $V_{PC}$, which is an approximate solution of \eqref{eq:ph-lift-geom}. It is natural to wonder whether, conversely, from a solution $V_{PL}$ of \eqref{eq:ph-lift-geom}, one can construct an approximate solution $V_{PC}$ of \eqref{eq:ph-cut-geom}. It does not seem to be the case. One could for example imagine setting $V_{PC}=\diag(R)V_{PL}\diag(R)$, where $R_i={b_i}/{\sqrt{V_{PL\,ii}}}$. Then $V_{PC}$ would not necessarily minimize \eqref{eq:ph-cut-geom} but at least belong to $\mathcal{H}_b$. But $\|V_{PC}-V_{PL}\|_2$ might be quite large: \eqref{eq:ph-lift-geom} implies that $\|\diag(V_{PL})-b^2\|_s$ is small but, if some coefficients of $b$ are very small, some $R_i$ may still be huge, so $\diag(R)\not\approx\idm$. This does happen in practice (see \S~\ref{ss:with_noise}).

To conclude this section, we relate this definition of stability to the one introduced in \citep{Cand12}. Suppose that $A$ is a matrix of random gaussian independant measurements such that $\Expect[|A_{i,j}|^2]=1$ for all $i,j$. We also suppose that $n\geq c_0 p$ (for some $c_0$ independent of $n$ and $p$). In the noisy setting, \citet{Cand12} showed that the minimizer $X$ of a modified version of~\ref{eq:ph-lift} satisfies with high probability
\begin{equation}\label{eq:candes_stability}
||X-x_0x_0^*||_2\leq C_0\frac{||\,|Ax_0|^2-b^2\,||_1}{n}
\end{equation}
for some $C_0$ independent of all variables. Assuming that the \ref{eq:weak-ph-lift} solution in~\eqref{eq:ph-lift-geom} behaves as~\ref{eq:ph-lift} in a noisy setting and that \eqref{eq:candes_stability} also holds for~\ref{eq:weak-ph-lift}, then
\begin{align*}
||AXA^*-(Ax_0)(Ax_0)^*||_2
&\leq ||A||_\infty^2||X-x_0x_0^*||_2\\
&\leq C_0\frac{||A||_\infty^2}{n}||\,|Ax_0|^2-b^2\,||_1\\
&\leq C_0\frac{||A||_\infty^2}{n}(2||Ax_0||_2+||b_{\rm noise}||_2)||b_{\rm noise}||_2
\end{align*}
Consequently, for any $C>2C_0\frac{||A||_\infty^2}{n}$, \ref{eq:weak-ph-lift} is $C$-stable in all $x_0$. With high probability, $||A||_\infty^2\leq (1+1/8)n$ (it is a corollary of \citep[Lemma 2.1]{Cand12}) so \ref{eq:weak-ph-lift} (and thus  also \ref{eq:ph-SDP}) is $C$-stable with high probability for some $C$ independent of all parameters of the problem.

\subsection{Perturbation Results}\label{ss:pert}
We recall here sensitivity analysis results for semidefinite programming from \cite{Todd99,Yild03}, which produce explicit bounds on the impact of small perturbations in the observation vector $b^2$ on the solution $V$ of the semidefinite program~\eqref{eq:scaled-pb}. Roughly speaking, these results show that if $b^2+b_{noise}$ remains in an explicit ellipsoid (called Dikin's ellipsoid), then interior point methods converge back to the solution in one full Newton step, hence the impact on $V$ is linear, equal to the Newton step. These results are more numerical in nature than the stability bounds detailed in the previous section, but they precisely quantify both the size and, perhaps more importantly, the geometry of the stability region.

\subsection{Complexity Comparisons}\label{ss:complexity}
Both the relaxation in~\ref{eq:ph-lift} and that in~\ref{eq:ph-SDP} are semidefinite programs and we highlight below the relative complexity of solving these problems depending on algorithmic choices and precision targets. Note that, in their numerical experiments, \citep{Cand11} solve a penalized formulation of problem~\ref{eq:ph-lift}, written
\BEQ\label{eq:pen-lift}
\min_{X\succeq 0}~\sum_{i=1}^n ( \Tr(a_ia_i^*X)-b_i^2)^2 + \lambda \Tr(X)
\EEQ
in the variable $X\in\herm_p$, for various values of the penalty parameter $\lambda>0$. 

The trace norm promotes a low rank solution, and solving a sequence of weighted trace-norm problems has been shown to further reduce the rank in \citep{Faze03,Cand11}. This method replaces $\Tr(X)$ by $\Tr (W_k X)$ where $W_0$ is initialized to the identity $I$. Given a solution $X_k$ of the resulting semidefinite program, the weighted matrix is updated to $W_{k+1} = (X_k + \eta I)^{-1}$ (see \citet{Faze03} for details). We denote by~$K$ the total number of such iterations, typically of the order of $10$. Trace minimization is not needed for the semidefinite program (\ref{eq:ph-SDP}), where the trace is fixed because we optimize over a normalized phase vector. However, weighted trace-norm iterations could potentially improve performance in~\ref{eq:ph-SDP} as well.
% AA: should we try weighted 

Recall that $p$ is the size of the signal and $n$ is the number of measured samples with $n = J p$ in the examples reviewed in Section \ref{s:numres}. In the numerical experiments in \citep{Cand11} as well as in this paper, $J = 3,4,5$. The complexity of solving the~\ref{eq:ph-SDP} and~\ref{eq:ph-lift} relaxations in~\ref{eq:ph-lift} using generic semidefinite programming solvers such as SDPT3 \citep{Toh96}, {\em without exploiting structure}, is given by
\[
O\left(J^{4.5}\,p^{4.5}\log\frac{1}{\epsilon}\right)
\quad \mbox{and} \quad
O\left(K\,J^{2}\, p^{4.5} \log\frac{1}{\epsilon}\right)
\]
for \ref{eq:ph-SDP} and \ref{eq:ph-lift} respectively \citep[\S\,6.6.3]{Bent01}. The fact that the constraint matrices have only one nonzero coefficient in~\ref{eq:ph-SDP} can be exploited (the fact that the constraints $a_ia_i^*$ are rank one in~\ref{eq:ph-lift} helps, but it does not modify the principal complexity term) so we get
\[
O\left(J^{3.5}\,p^{3.5} \log\frac{1}{\epsilon}\right)
\quad \mbox{and} \quad
O\left(K\,J^{2} p^{4.5} \log\frac{1}{\epsilon}\right)
\]
for \ref{eq:ph-SDP} and \ref{eq:ph-lift} respectively using the algorithm in \citet{Helm96} for example. If we use first-order solvers such as TFOCS \citep{Beck12}, based on the optimal algorithm in \citep{Nest83}, the dependence on the dimension can be further reduced, to become
\[
O\left(\frac{J^3\,p^3}{\epsilon}\right)
\quad \mbox{and} \quad
O\left(\frac{K J\,p^3}{\epsilon}\right)
\]
for solving a penalized version of the \ref{eq:ph-SDP} relaxation and the penalized formulation of \ref{eq:ph-lift} in~\eqref{eq:pen-lift}. While the dependence on the signal dimensions $p$ is somewhat reduced, the dependence on the target precision grows from $\log(1/\epsilon)$ to $1/\epsilon$. Finally, the iteration complexity of the block coordinate descent Algorithm~\ref{alg:block} is substantially lower and its convergence is linear, but no fully explicit bounds on the number of iterations are known in our case. The complexity of the method is then bounded by $O\left(\log\frac{1}{\epsilon}\right)$ but the constant in this bound depends on $n$ here, and the dependence cannot be quantified explicitly.

Algorithmic choices are ultimately guided by precision targets. If $\epsilon$ is large enough so that a first-order solver or a block coordinate descent can be used, the complexity of \ref{eq:ph-SDP} is not significantly better than that of \ref{eq:ph-lift}. On the contrary, when $\epsilon$ is small, we must use an interior point solver, for which \ref{eq:ph-SDP}'s complexity is an order of magnitude lower than that of \ref{eq:ph-lift} because its constraint matrices are singletons. In practice, the target value for $\epsilon$ strongly depends on the sampling matrix $A$. For example, when $A$ corresponds to the convolution by $6$ Gaussian random filters (\S\ref{ss:random_filters}), to reconstruct a Gaussian white noise of size $64$ with a relative precision of $\eta$, we typically need $\epsilon\sim 2.10^{-1}\eta$. For $4$ Cauchy wavelets (\S\ref{ss:cauchy_wavelets}), it is twenty times less, with $\epsilon\sim 10^{-2}\eta$. For other types of signals than Gaussian white noise, we may even need~$\epsilon\sim 10^{-3}\eta$.

\subsection{Greedy Refinement}\label{ss:completing}
If the \ref{eq:ph-SDP} or \ref{eq:ph-lift} algorithms do not return a rank one matrix then an approximate solution of the phase recovery problem is obtained by extracting a leading eigenvector $v$. For \ref{eq:ph-SDP} and \ref{eq:ph-lift}, $\tilde x = A^\dag\diag(b) v$ and $\tilde x = v$ are respectively approximate solutions of the phase recovery problem with $|A \tilde x| \neq b = |A x|$. This solution is then refined by applying the \ref{alg:GS} algorithm initialized with $\tilde x$. If $\tilde x$ is sufficiently close to $x$ then, according to numerical experiments of Section \ref{s:numres}, this greedy algorithm converges to $\lambda\,x$ with $|\lambda| = 1$. These greedy iterations require much less operations than \ref{eq:ph-SDP} and \ref{eq:ph-lift} algorithms, and thus have no significant contribution to the computational complexity.

\subsection{Sparsity}
Minimizing $\Tr(X)$ in the \ref{eq:ph-lift} problem means looking for signals which match the modulus constraints and have minimum $\ell_2$ norm. In some cases, we have a priori knowledge that the signal we are trying to reconstruct is sparse, i.e. $\Card(x)$ is small. The effect of imposing sparsity was studied in e.g. \citep{Mora07,Oshe12,Li12}. 

Assuming $n \leq p$, the set of solutions to $\|Ax-\diag(b)u\|_2$ is written $x=A^\dag \diag(b) u + Fv$ where $F$ is a basis for the nullspace of $A$. The reconstruction problem with a $\ell_1$ penalty promoting sparsity is then written
\[\BA{ll}
\mbox{minimize} &  \|AA^\dag\diag(b)u-\diag(b)u\|_2^2 + \gamma \|A^\dag \diag(b) u + Fv\|_1^2 \\
\mbox{subject to} & |u_i|=1,
\EA\]
in the variables $u\in\complex^p$ and $y\in\complex^{p-n}$. Using the fact that $\|y\|_1^2=\|yy^*\|_{\ell_1}$, this can be relaxed as
\[\BA{ll}
\mbox{minimize} &  \Tr(UM_3) + \gamma \|VUV^*\|_{\ell_1} \\
\mbox{subject to} & U \succeq 0,\,|U_{ii}|=1, \quad i=1,\ldots,n,
\EA\]
which is a semidefinite program in the (larger) matrix variable $U\in\herm_{p}$, with $M_3$ the matrix containing~$M$ in its upper left block and zeros elsewhere, and $V=(A^\dag\diag(b),F)$. 

On the other hand, when $n > p$ and $A$ is injective, the matrix $F$ disappears, $M_3=M$ and taking sparsity into account simply amounts to adding an~$l_1$ penalization to~\ref{eq:ph-SDP}. As noted in~\citep{Voro12a} however, the effect of an $\ell_1$ penalty on least-squares solutions is not completely clear.

\vskip 6ex
\section{Numerical results} \label{s:numres}
In this section, we compare the numerical performance of the \ref{alg:GS} (greedy), \ref{eq:ph-SDP} and \ref{eq:ph-lift} algorithms on various phase recovery problems. As in \citep{Cand11}, the \ref{eq:ph-lift} problem is solved using the package in \citep{Beck12}, with reweighting, using $K = 10$ outer iterations and $1000$ iterations of the first order algorithm. The \ref{eq:ph-SDP} and \ref{alg:GS} algorithms described here are implemented in a public software package available at 
\begin{center}
\small
\url{http://www.cmap.polytechnique.fr/scattering/code/phaserecovery.zip}
\end{center}
These phase recovery algorithms computes an approximate solution $\tilde x$ from $|A x|$ and the reconstruction error is measured by the relative Euclidean distance up to a complex phase given by
\begin{equation}
\label{errorsig}
\epsilon(x,\tilde x) \triangleq \min_{c \in \C,|c| = 1}\, \frac{\|x - c\, \tilde x\|} {\|x\|}.
\end{equation}
We also record the error over measured amplitudes, written
\begin{equation}
\label{errorsig2}
\epsilon(|A x| ,|A \tilde x|) 
\triangleq \frac{\|| A x| - |A \tilde x|\|} {\|A x\|}.
\end{equation}
Note that when the phase recovery problem either does not admit a unique solution or is unstable, we usually have $\epsilon (|Ax|,|A\tilde x|) \ll \epsilon(x,\tilde x)$. In the next three subsections, we study these reconstruction errors for three different phase recovery problems, where $A$ is defined as an oversampled Fourier transform, as multiple filterings with random filters, or as a wavelet transform. Numerical results are computed on three different types of test signals $x$: realizations of a complex Gaussian white noise, sums of complex exponentials~$a_\omega \, e^{i \omega m}$ with random frequencies $\omega$ and random amplitudes $a_\omega$ (the number of exponentials is random, around $6$), and signals whose real and imaginary parts are scan-lines of natural images. Each signal has $p=128$ coefficients. Figure~\ref{three_classes} shows the real part of sample signals, for each signal type.

\begin{figure}
\begin{tabular}{ccc}
\includegraphics[width=.33\textwidth]{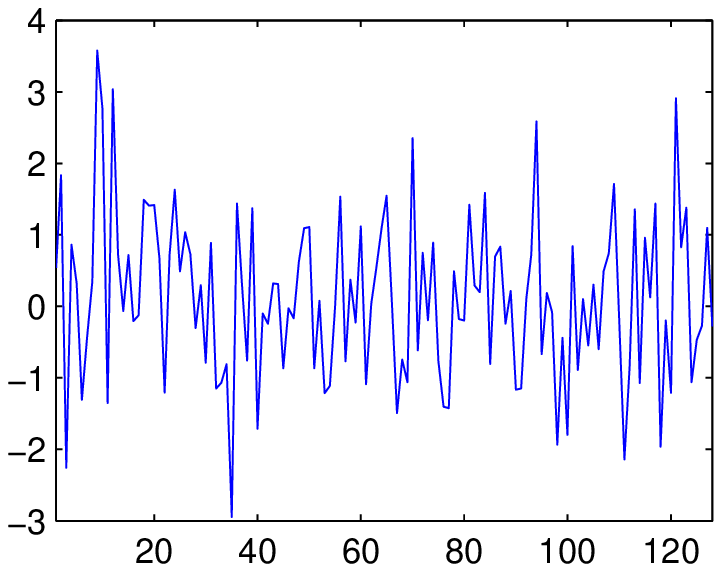} &
\includegraphics[width=.33\textwidth]{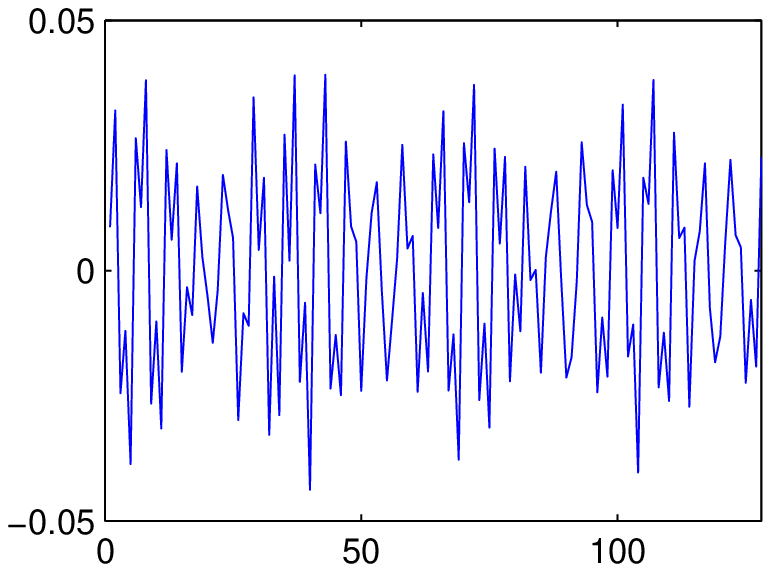}&
\includegraphics[width=.33\textwidth]{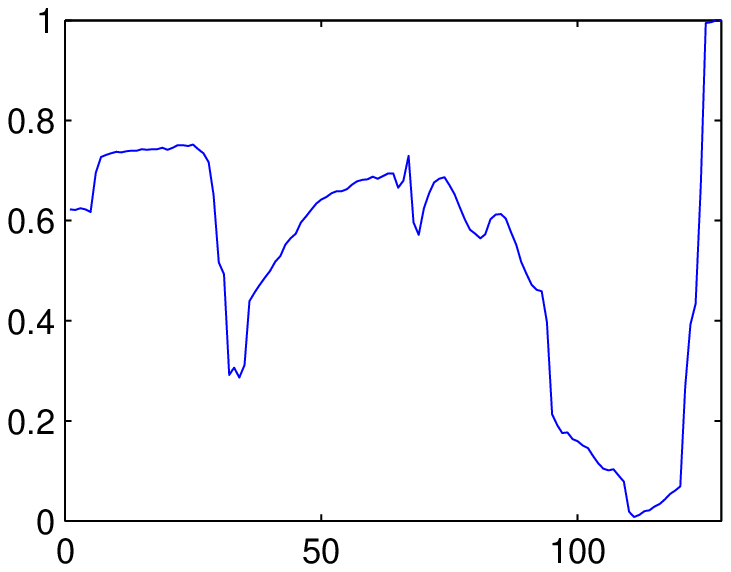}\\
(a) & (b) & (c)
\end{tabular}
\caption{Real parts of sample test signals. (a)~Gaussian white noise. (b)~Sum of 6 sinuoids of random frequency and random amplitudes. (c)~Scan-line of an image.}
\label{three_classes}
\end{figure}

\subsection{Oversampled Fourier Transform}

The discrete Fourier transform $\hat y$ of a signal $y$ of $q$ coefficients is written
\[
\hat y_k = \sum_{m=0}^{q-1} y_m\exp(\frac{-i 2 \pi  k m}{q})~.
\]
In X-ray crystallography or diffraction imaging experiments, compactly supported signals are estimated from the amplitude of Fourier transforms oversampled by a factor $J \geq 2$. The corresponding operator $A$ computes an oversampled discrete Fourier transform evaluated over $n = J p$ coefficients. The signal $x$ of size $p$ is extended into $x^J$ by adding $(J-1) p$ zeros and
\[
(A x)_k = \hat x^J_k = \underset{m=1}{\overset{p}{\sum}} x_m\exp(-\frac{i 2 \pi  k m}{n}).
\]
For this oversampled Fourier transform, the phase recovery problem does not have a unique solution
\citep{Akut56}. In fact, one can show \citep{sanz} that there are as many as $2^{p-1}$ solutions $\tilde x \in\C^p$ such that $|A \tilde x| = |A x|$. Moreover, increasing the oversampling factor $J$ beyond $2$ does not reduce the number of solutions. 

Because of this intrinsic instability, we will observe that all algorithms perform similarly on this type of reconstruction problems and Table~\ref{percentage} shows that the percentage of perfect reconstruction is below $5\%$ for all methods. The signals which are perfectly recovered are sums of few sinusoids. Because these test signals are very sparse in the Fourier domain, the number of signals having identical Fourier coefficient amplitudes is considerably smaller than in typical sample signals. As a consequence, there is a small probability (about~5\%) of exactly reconstructing the original signal given an arbitrary initialization. None of the Gaussian random noises and image scan lines are exactly recovered. Note that we say that an exact reconstruction is reached when $\epsilon(x,\tilde x) < 10^{-2}$ because a few iterations of the \ref{alg:GS} algorithm from such an approximate solution $\tilde x$ will typically converges to $x$. Numerical results are computed with 100 sample signals in each of the 3 signal classes. 

Table~\ref{error1} gives the average relative error $\epsilon(x,\tilde x)$ over signals that are not perfectly reconstructed, which is of order one here. Despite this large error, Table \ref{error2} shows that the relative error $\epsilon(|Ax|,|A\tilde x|)$ over the Fourier modulus coefficients is below $10^{-3}$ for all algorithms. This is due to the non-uniqueness of the phase recovery from Fourier modulus coefficients. 
Recovering a solution $\tilde x$ with identical or nearly identical oversampled
Fourier modulus coefficients as $x$ does not
guarantee that~$\tilde x$ is proportional to~$x$.
Overall, in this set of ill-posed Fourier experiments, recovery performance is very poor for all methods and the \ref{eq:ph-SDP} and \ref{eq:ph-lift} relaxations do not improve much on the results of the faster \ref{alg:GS} algorithm.

\begin{table}
\begin{tabular}{|c|c|c|c|}\hline
& Fourier & Random Filters & Wavelets\\\hline
\ref{alg:GS}&5\%&49\%&0\%\\\hline
\ref{eq:ph-lift} with reweighting&3\%&100\%&62\%\\\hline
\ref{eq:ph-SDP} &4\%&100\%&100\%\\\hline
\end{tabular}
\vskip0.5ex
\caption{Percentage of perfect reconstruction from $|Ax|$,
over 300 test signals, 
for the three different operators $A$ (columns) and the three algorithms (rows).
\label{percentage}}
\end{table}

\begin{table}
\begin{tabular}{|c|c|c|c|}\hline
& Fourier & Random Filters & Wavelets\\\hline
\ref{alg:GS}&0.9&1.2&1.3\\\hline
\ref{eq:ph-lift} with reweighting&0.8&exact&0.5\\\hline
\ref{eq:ph-SDP} &0.8&exact&exact\\\hline
\end{tabular}
\vskip0.5ex
\caption{Average relative signal reconstruction error $\epsilon(\tilde x , x)$
over all test signals that are not perfectly reconstructed,
for each operator $A$ and each algorithm. \label{error1}}

\end{table}

\begin{table}
\begin{tabular}{|c|c|c|c|}\hline
& Fourier & Random Filters & Wavelets\\\hline
\ref{alg:GS}&$9.10^{-4}$&0.2&0.3\\\hline
\ref{eq:ph-lift} with reweighting&$5.10^{-4}$&exact&$8.10^{-2}$\\\hline
\ref{eq:ph-SDP} &$6.10^{-4}$&exact&exact\\\hline
\end{tabular}
\vskip0.5ex
\caption{Average relative 
error $\epsilon(|A\tilde x| , |Ax|)$ of coefficient amplitudes,
over all test signals that are not perfectly reconstructed,
for each operator $A$ and each algorithm.\label{error2}}
\end{table}

\subsection{Multiple Random Illumination Filters\label{ss:random_filters}}

To guarantee uniqueness of the phase recovery problem, one can add independent measurements by ``illuminating'' the object through J filters $h^j$ in the context of X-ray imaging or crystallography \citep{Cand11a}. The resulting operator $A$ is the discrete Fourier transform of $x$ multiplied by each filter $h^j$ of size $p$
\[
(A x)_{k + p j} = \widehat {(x h^j)}_k  = (\widehat x \star \hat h^j)_k  \quad
\mbox{for $1 \leq j \leq J$ and $0 \leq k < p$,}
\]
where $\hat x \star \hat h^j $ is the circular convolution between $\hat x$ and $\hat h^j$. \citet{Cand11a} proved that, for some constant $C>0$ large enough, $Cp$ Gaussian independent measurements are sufficient to perfectly reconstruct a signal of size $p$, with high probability. Similarly, we would expect that, picking the filters $h^j$ as realizations of independent Gaussian random variables, perfect reconstruction will be guaranteed with high probability if $J$ is large enough (and independent of $p$). This result has not yet been proven because Gaussian filters do not give independent measurements but \citet{Cand11} observed that, empirically, for signals of size $p = 128$, with $J = 4$ filters, perfect recovery is achieved in $100\%$ of experiments.

Table~\ref{percentage} confirms this behavior and shows that the \ref{eq:ph-SDP} algorithm achieves perfect recovery in all our experiments. As predicted by the equivalence results presented in the previous section, we observe that \ref{eq:ph-SDP} and \ref{eq:ph-lift} have identical performance in these experiments. With $4$ filters, the solutions of  these two SDP relaxations are not of rank one but are ``almost'' of rank one, in the sense that their first eigenvector~$v$ has an eigenvalue much larger than the others, by a factor of about $5$ to $10$. Numerically, we observe that the corresponding approximate solutions, $\tilde x = \diag(v) b$, yield a relative error $\epsilon(|A x| , |A \tilde x|)$ which, for scan-lines of images and especially for Gaussian signals, is of the order of the ratio between the largest and the second largest eigenvalue of the matrix $U$. The resulting solutions $\tilde x$ are then sufficiently close to $x$ so that a few iterations of the \ref{alg:GS} algorithm started at $\tilde x$ will converge to $x$.

% Exemples de rapports des deux plus grandes valeurs propres : 8.3, 7.6, 10.7, 11.8, 4.7, 5.2, 9.9, 6.2, 8.0, 7.5, 9.6, 8.9

Table \ref{percentage} shows however that directly applying the \ref{alg:GS} algorithm starting from a random initialization point yields perfect recovery in only about $50\%$ of our experiments. This percentage decreases as the signal size $p$ increases. The average error $\epsilon(x , \tilde x)$ on non-recovered signals in Table \ref{error1} is $1.3$ whereas on the average error on the modulus $\epsilon(|Ax| , |A\tilde x|)$ is $0.2$.

\subsection{Wavelet Transform{\label{ss:cauchy_wavelets}}}

Phase recovery problems from the modulus of wavelet coefficients appear in audio signal processing where this modulus is used by many audio and speech recognition systems. These moduli also provide physiological models of cochlear signals in the ear \citep{shamma} and recovering audio signals from wavelet modulus coefficients is an important problem in this context. 

To simplify experiments, we consider wavelets dilated by dyadic factors $2^j$ which have a lower frequency resolution than audio wavelets. A discrete wavelet transform is computed by circular convolutions with discrete wavelet filters, i.e.
\[
(A x)_{k+j p} = (x \star \psi^j)_k = \sum_{m=1}^{p} x_m \psi^j_{k-m}\quad
\mbox{for $1 \leq j \leq J-1$ and $1 \leq k \leq p$}
\]
where $\psi^j_ m$  is a $p$ periodic wavelet filter. It is defined by dilating, sampling and periodizing a complex wavelet $\psi \in {\bf L^2}(\complex)$, with
\[
\psi^j_ m = \sum_{k=-\infty}^{\infty} \psi (2^{j} (m / p - k))\quad
\mbox{for $1 \leq m \leq p$}.
\]
Numerical computations are performed with a Cauchy wavelet whose Fourier transform is
\[
\hat \psi(\omega) = \omega^d\, e^{- \omega}\, {\bf 1}_{\omega > 0},
\]
up to a scaling factor, with $d = 5$. To guarantee that $A$ is an invertible operator, the lowest signal frequencies are carried by a suitable low-pass filter $\phi$ and
\[
(A x)_{k+J p} = (x \star \phi)_k \quad
\mbox{for $1 \leq k \leq p$.}
\]
One can prove that $x$ is always uniquely determined by $|Ax|$, up to a multiplication factor. When $x$ is real, the reconstruction appears to be numerically stable. Recall that the results of \S\ref{sss:real_valued_signal} allow us to explicitly impose the condition that $x$ is real in the \ref{eq:ph-SDP} recovery algorithm. For \ref{eq:ph-lift} in \citet{Cand11}, this condition is enforced by imposing that $X = x x^*$ is real. For the \ref{alg:GS} algorithm, when $x$ is real, we simply project at each iteration on the image of $\R^p$ by $A$, instead of projecting on the image of $\C^p$ by $A$.

Numerical experiments are performed on the real part of the complex test signals. Table \ref{percentage} shows that \ref{alg:GS} does not reconstruct exactly any real test signal from the modulus of its wavelet coefficients. The average relative  error $\epsilon(\tilde x , x)$ in Table \ref{error1} is $1.2$ where the coefficient amplitudes have an average error $\epsilon(|A\tilde x| , |Ax|)$ of $0.3$ in Table \ref{error2}.

\ref{eq:ph-lift} reconstructs $62\%$ of test signals, but the reconstruction rate varies with signal type. 
The proportions of exactly reconstructed signals among random noises,
sums of sinusoids and image scan-lines are  $27\%$, $60\%$ and $99\%$ respectively. Indeed, image scan-lines have 
a large proportion of wavelet coefficients whose amplitudes are negligible.
The proportion of phase coefficients having a strong impact on the reconstruction
of $x$ is thus much smaller for scan-line images than for random noises,
which reduces the number of significant variables to recover.
Sums of sinuoids of random frequency have wavelet coefficients whose sparsity
is intermediate between image scan-lines and Gaussian white noises, which
explains the intermediate performance of \ref{eq:ph-lift} on these signals.
The overall average error $\epsilon(\tilde x , x)$ on non-reconstructed signals
is $0.5$. Despite this relatively important
error, $\tilde x$ and $x$ are usually almost equal on most of their support, up to a sign switch, and the importance of the error is precisely due to the local phase inversions (which change signs).

The \ref{eq:ph-SDP} algorithm reconstructs exactly all test signals. Moreover, the recovered matrix $U$ is always of rank one and it is therefore not necessary to refine the solution with \ref{alg:GS} iterations. At first sight, this difference in performance between \ref{eq:ph-SDP} and \ref{eq:ph-lift} may seem to contradict the equivalence results of \S\ref{ss:relaxation_tightness} (which are valid when $x$ is real and when $x$ is complex). It can be explained however by the fact that $10$ steps of reweighing and $1000$ inner iterations per step are not enough to let \ref{eq:ph-lift} fully converge. In these experiments, the precision required to get perfect reconstruction is very high and, consequently, the number of first-order iterations required to achieve it is too large (see \S\ref{ss:complexity}). With an interior-point-solver, this number would be much smaller but the time required per iteration would become prohibitively large. The much simpler structure of the~\ref{eq:ph-SDP} relaxation allows us to solve these larger problems more efficiently.

\subsection{Impact of Trace Minimization\label{ss:role_of_trace}}
We saw in~\S\ref{ss:weak_version} that, in the absence of noise, \ref{eq:ph-SDP} was very similar to a simplified version of PhaseLift, \ref{eq:weak-ph-lift}, in which no trace minimization is performed. Here, we confirm empirically that \ref{eq:weak-ph-lift} and \ref{eq:ph-lift} are essentially equivalent. Minimizing the trace is usually used as rank minimization heuristic, with recovery guarantees in certain settings \citep{Faze03,Cand08,Chan12} but it does not seem to make much difference here. In fact, \citet{Dema12} recently showed that in the independent experiments setting, \ref{eq:weak-ph-lift} has a unique (rank one) solution with high probability, i.e. the feasible set of~\ref{eq:ph-lift} is a singleton and trace minimization has no impact. Of course, from a numerical point of view, solving the feasibility problem~\ref{eq:weak-ph-lift} is about as hard as solving the trace minimization problem~\ref{eq:ph-lift}, so the result \citep{Dema12} simplifies analysis but does not really affect numerical performance.

Figure \ref{weak} compares the performances of \ref{eq:ph-lift} and \ref{eq:weak-ph-lift} as a function of~$n$ (the number of measurements). We plot the percentage of successful reconstructions ({\em left}) and the percentage of cases where the relaxation was exact, i.e. the reconstructed matrix $X$ was rank one ({\em right}). The plot shows a clear phase transitions when the number of measurements increases. For \ref{eq:ph-lift}, these transitions happen respectively at $n=155\approx 2.5p$ and $n=285\approx 4.5p$, while for \ref{eq:weak-ph-lift}, the values become $n=170\approx 2.7p$ and $n=295\approx 4.6p$, so the transition thresholds are very similar. Note that, in the absence of noise, \ref{eq:weak-ph-lift} and \ref{eq:ph-SDP} have the same solutions, up to a linear transformation (see~\S\ref{ss:geometric_point_of_view}), so we can expect the same behavior in the comparison \ref{eq:ph-SDP} versus \ref{eq:ph-SDP2}.

\begin{figure}
\begin{tabular}{cc}
\psfrag{x}[t][b]{Number of measurements}
\psfrag{y}[b][t]{Reconstruction rate}
\psfrag{Weak PhaseLift}{\footnotesize Weak PhaseLift}
\psfrag{PhaseLift}{\footnotesize PhaseLift}
\includegraphics[width=.45\textwidth]{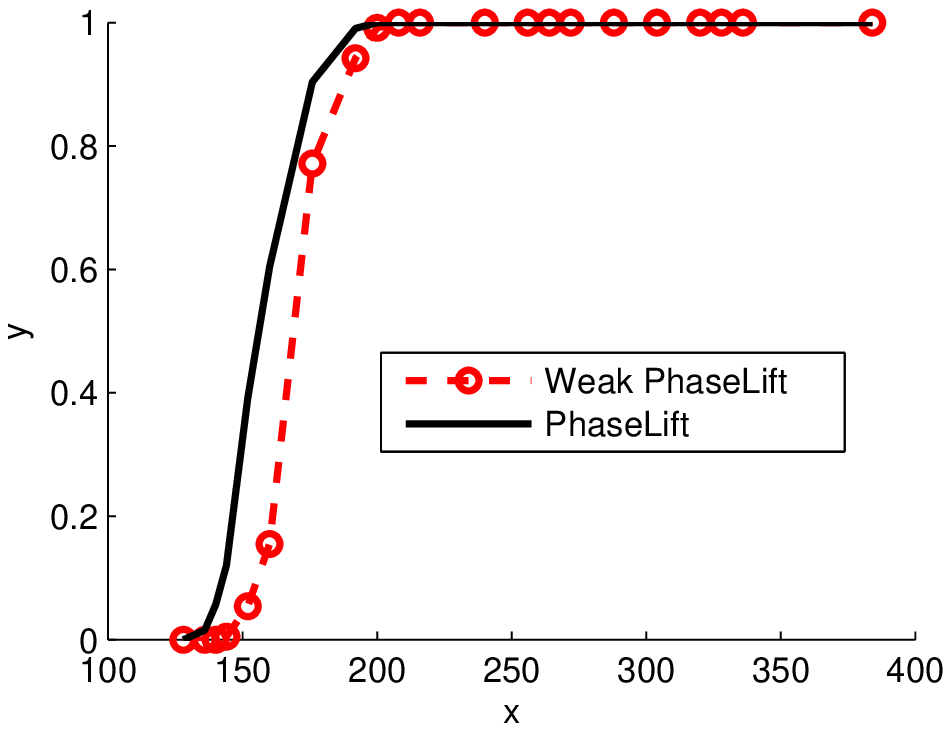} &
\psfrag{x}[t][b]{Number of measurements}
\psfrag{y}[b][t]{Proportion of rank one sols.}
\psfrag{Weak PhaseLift}{\footnotesize Weak PhaseLift}
\psfrag{PhaseLift}{\footnotesize PhaseLift}
\includegraphics[width=.45\textwidth]{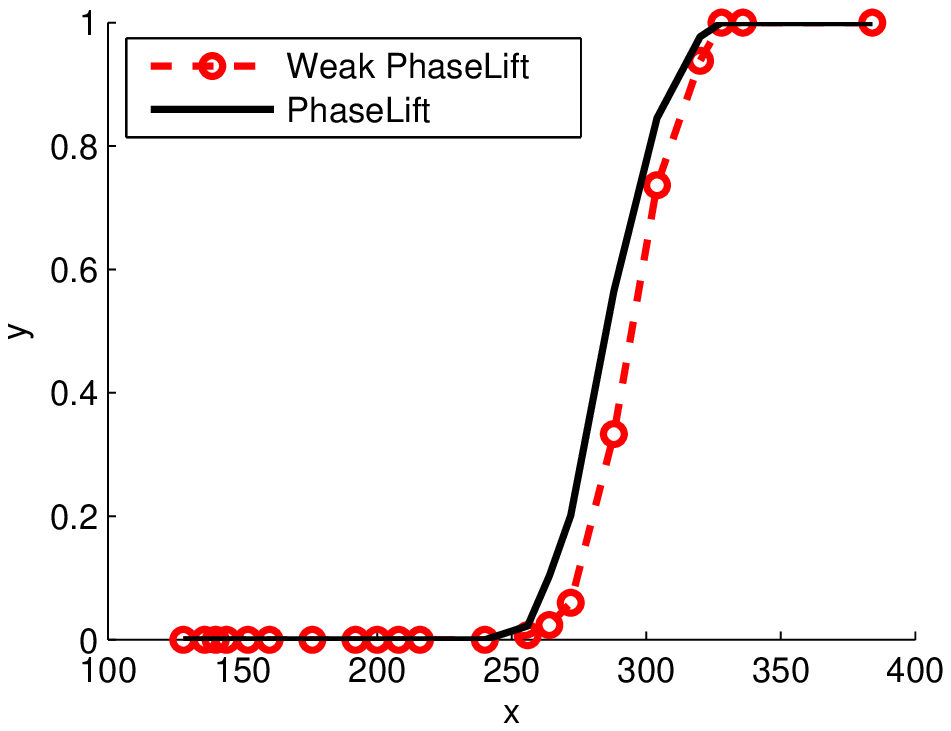}
\end{tabular}
\caption{Comparison of \ref{eq:ph-lift} and \ref{eq:weak-ph-lift} performance, for $64$-sized signals, as a function of the number of measurements. Reconstruction rate, after \ref{alg:GS} iterations {\em (left)} and proportion of rank one solutions {\em (right)}. \label{weak}}
\end{figure}

\subsection{Reconstruction in the Presence of Noise\label{ss:with_noise}}
Numerical stability is crucial for practical applications. In this last subsection, we suppose that the vector $b$ of measurements is of the form
\begin{equation*}
b = |Ax| + b_{\rm noise}
\end{equation*}
with $\|b_{\rm noise}\|_2=o(\|Ax\|_2)$. In our experiments, $b_{\rm noise}$ is always a Gaussian white noise. Two reasons can explain numerical instabilities in the solution $\tilde x$. First, the reconstruction problem itself can be unstable, with $\|\tilde x-cx\|\gg\||A\tilde x|-|Ax|\|$ for all $c\in\C$. Second, the algorithm may fail to reconstruct $\tilde{x}$ such that $\||A\tilde{x}|-b\| \approx \|b_{\rm noise}\|$. No algorithm can overcome the first cause but good reconstruction methods will overcome the second one. In the following paragraphs, to complement the results in \S\ref{ss:noisy_measurements}, we will demonstrate empirically that \ref{eq:ph-SDP} is stable, and compare its performances with \ref{eq:ph-lift}. We will observe in particular that \ref{eq:ph-SDP} appears to be more stable than \ref{eq:ph-lift} when $b$ is sparse.

\subsubsection{Wavelet transform}
Figure \ref{bruit_cw4_gs} displays the performance of \ref{eq:ph-SDP} in the wavelet transform case. It shows that \ref{eq:ph-SDP} is stable up to around $5-10\%$ of noise. Indeed, the reconstructed $\tilde{x}$ usually satisfies $\epsilon(|Ax|,|A\tilde x|)=\|\,|Ax|-|A\tilde{x}|\,\|_2\leq\|b_{\rm noise}\|_2$, which is the best we can hope for.  Wavelet transform is a case where the underlying phase retrieval problem may present instabilities, therefore the reconstruction error $\epsilon(x,\tilde x)$ is sometimes much larger than $\epsilon(|Ax|,|A\tilde x|)$. This remark applies especially to sums of sinusoids, which represent the most unstable case.

When all coefficients of $Ax$ have approximately the same amplitude, \ref{eq:ph-lift} and \ref{eq:ph-SDP} produce similar results, but when $Ax$ is sparse, \ref{eq:ph-lift} appears less stable. We gave a qualitative explanation of this behavior at the end of \S\ref{ss:noisy_measurements} which seems to be confirmed by the results in Figure \ref{bruit_cw4_gs}. This boils down to the fact that the values of the phase variables in~\ref{eq:ph-SDP} corresponding to zeros in $b$ can be set to zero so the problem becomes much smaller. Indeed, the performance of \ref{eq:ph-lift} and \ref{eq:ph-SDP} are equivalent in the case of Gaussian random filters (where measurements are never sparse), they are a bit worse in the case of sinusoids (where measurements are sometimes sparse) and quite unsatisfactory for scan-lines of images (where measurements are always sparse).

\begin{figure}
\begin{tabular}{cc}
\bf Gaussian & \bf Sinusoids\\
\psfrag{xx}[t][b]{Amount of noise}
\psfrag{yy}[b][t]{Relative error}
\includegraphics[width=.4\textwidth]{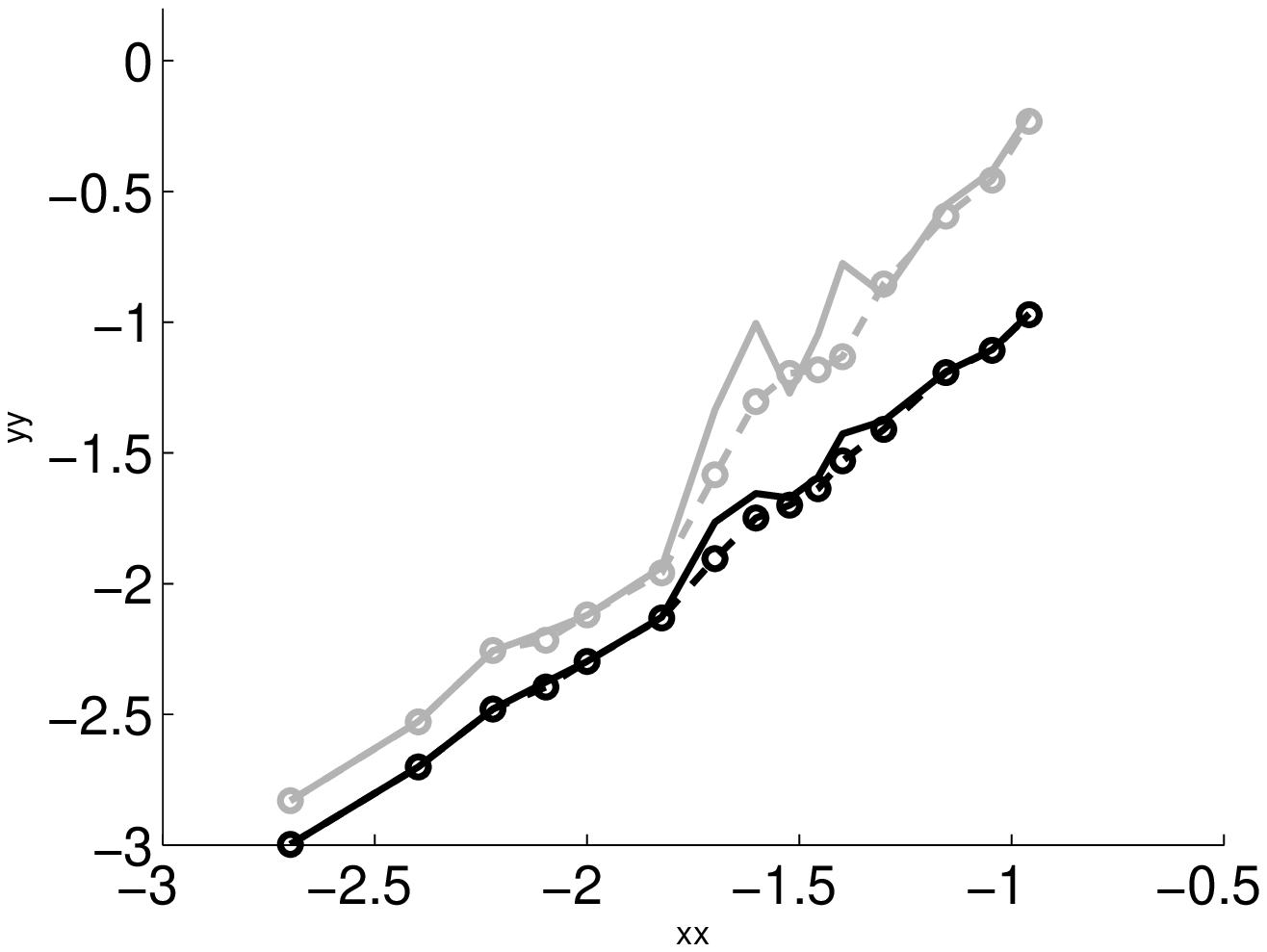}&
\psfrag{xx}[t][b]{Amount of noise}
\psfrag{yy}[b][t]{Relative error}
\includegraphics[width=.4\textwidth]{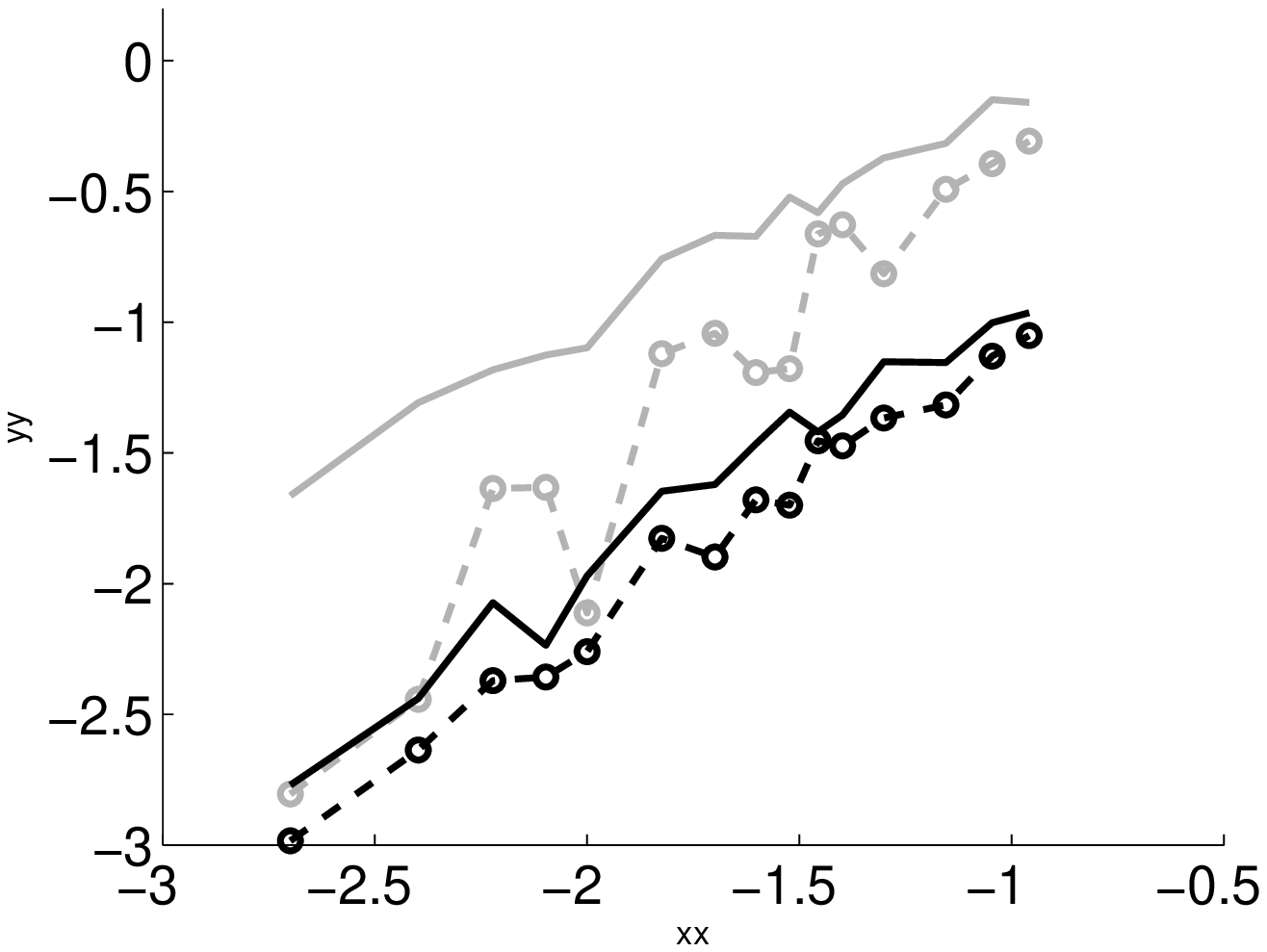}\\
\textbf{ }\\
\bf Image Scan-Lines & \\
\psfrag{xx}[t][b]{Amount of noise}
\psfrag{yy}[b][t]{Relative error}
\includegraphics[width=.4\textwidth]{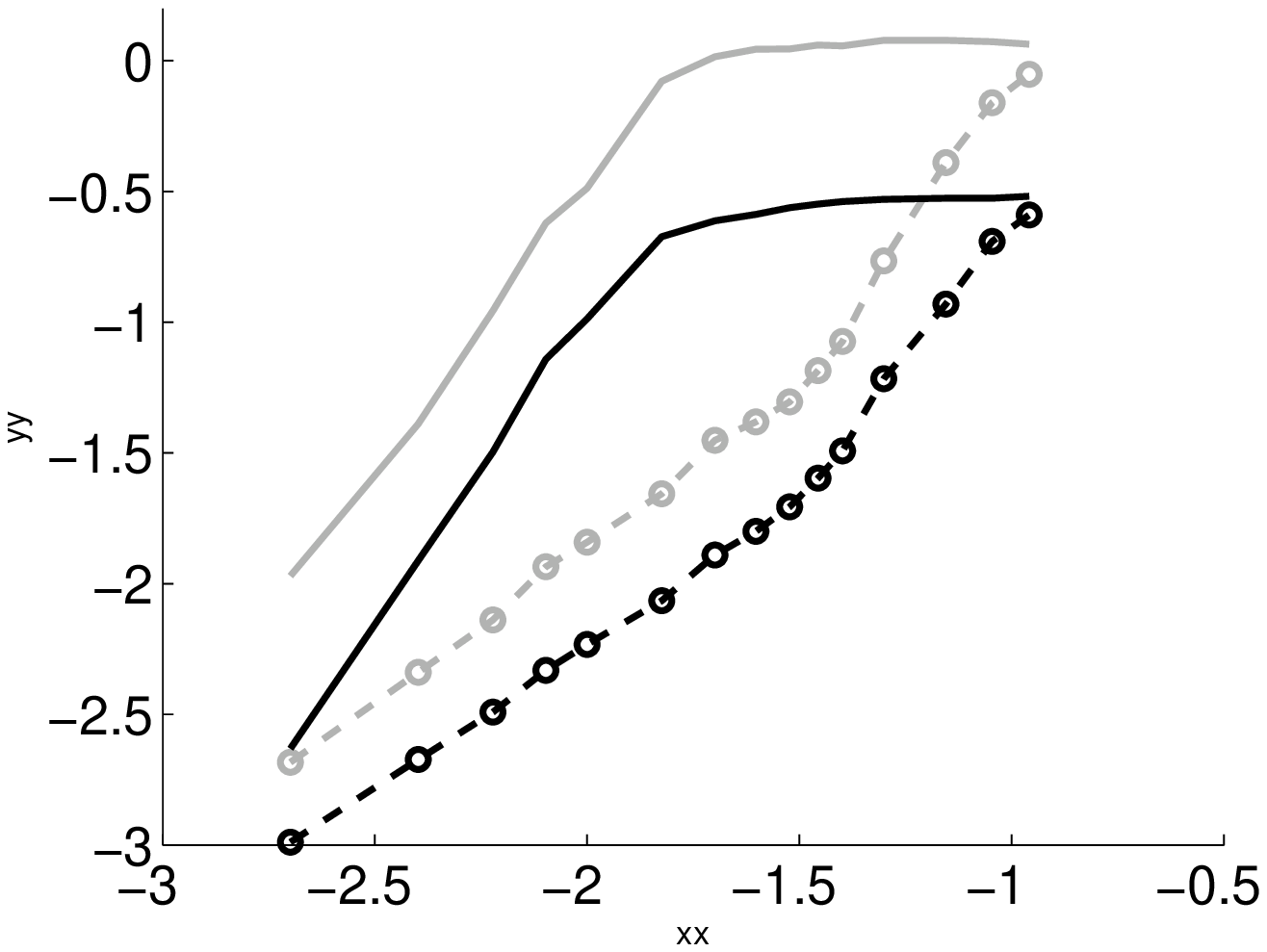}&
\begin{minipage}{0.35\textwidth}
\vskip -30ex
\includegraphics[width=\textwidth]{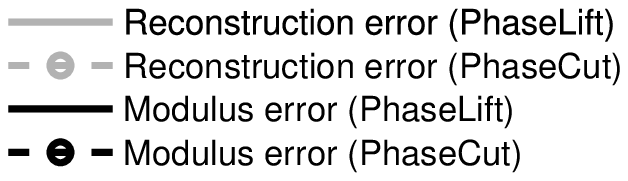}
\end{minipage}
\end{tabular}
\caption{Mean reconstruction errors versus amount of noise for \ref{eq:ph-lift} and \ref{eq:ph-SDP}, both in decimal logarithmic scale, for three types of signals: Gaussian white noises, sums of sinusoids and scan-lines of images. Both algorithms were followed by a few hundred \ref{alg:GS} iterations. \label{bruit_cw4_gs}}
\end{figure}

\subsubsection{Multiple random illumination filters}

\citet{Cand12} proved that, if $A$ was a Gaussian matrix, the reconstruction problem was stable with high probability, and \ref{eq:ph-lift} reconstructed a $\tilde{x}$ such that
\begin{equation*}
\epsilon(\tilde{x},x)\leq O\left(\frac{\|b_{\rm noise}\|_2}{\|Ax\|_2}\right).
\end{equation*}
The same result seems to hold for $A$ corresponding to Gaussian random illumination filters (cf. \S\ref{ss:random_filters}). Moreover, \ref{eq:ph-SDP} is as stable as \ref{eq:ph-lift}. Actually, up to $20\%$ of noise, when followed by some \ref{alg:GS} iterations, \ref{eq:ph-SDP} and \ref{eq:ph-lift} almost always reconstruct the same function. Figure~\ref{bruit_grf4_gs} displays the corresponding empirical performance, confirming that both algorithms are stable. The relative reconstruction errors are approximately linear in the amount of noise, with
\begin{equation*}
\epsilon(|A\tilde{x}|,|Ax|)\approx 0.8\times\frac{\|b_{\rm noise}\|_2}{\|Ax\|_2}
\quad \mbox{and} \quad
\epsilon(\tilde{x},x)\approx 2\times\frac{\|b_{\rm noise}\|_2}{\|Ax\|_2}
\end{equation*}
in our experiments.

\begin{figure}
\begin{tabular}{ccc}
\psfrag{xx}[t][b]{Amount of noise}
\psfrag{yy}[b][t]{Relative error}
\psfrag{Reconstruction error (PhaseLift)}{\tiny Reconstr. error (PhaseLift)}
\psfrag{Reconstruction error (PhaseCut)}{\tiny Reconstr. error (PhaseCut)}
\psfrag{Modulus error (PhaseLift)}{\tiny Modulus error (PhaseLift)}
\psfrag{Modulus error (PhaseCut)}{\tiny Modulus error (PhaseCut)}
\includegraphics[width=.4\textwidth]{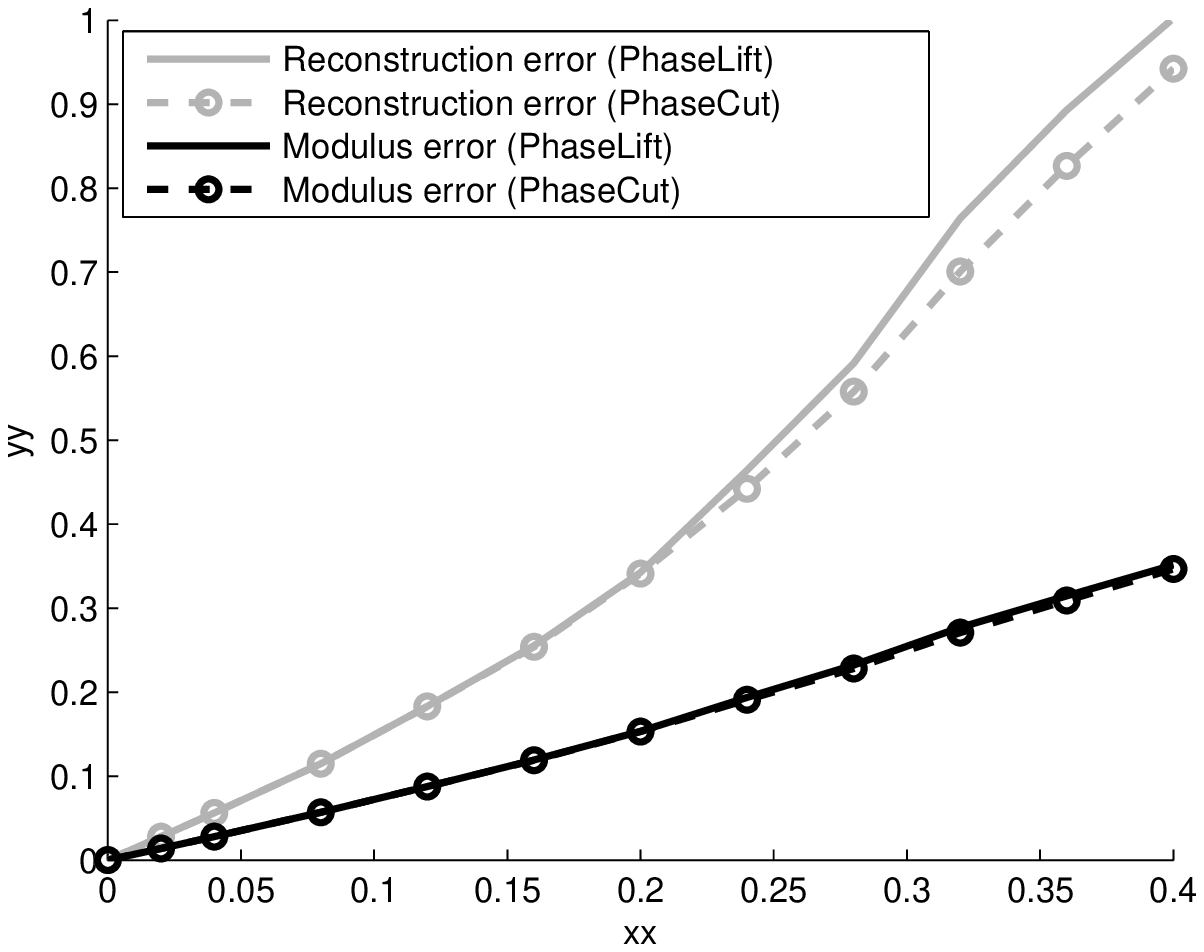} & &
\psfrag{y}[b][t]{Reconstruction error}
\psfrag{x}[t][b]{Amount of noise}
\psfrag{PhaseLift (for non-sparse measurements)}{\tiny PhaseLift (non-sparse)}
\psfrag{PhaseCut (for non-sparse measurements)}{\tiny PhaseCut (non-sparse)}
\psfrag{PhaseLift (for sparse measurements)}{\tiny PhaseLift (sparse)}
\psfrag{PhaseCut (for sparse measurements)}{\tiny PhaseCut (sparse)}
\includegraphics[width=.4\textwidth]{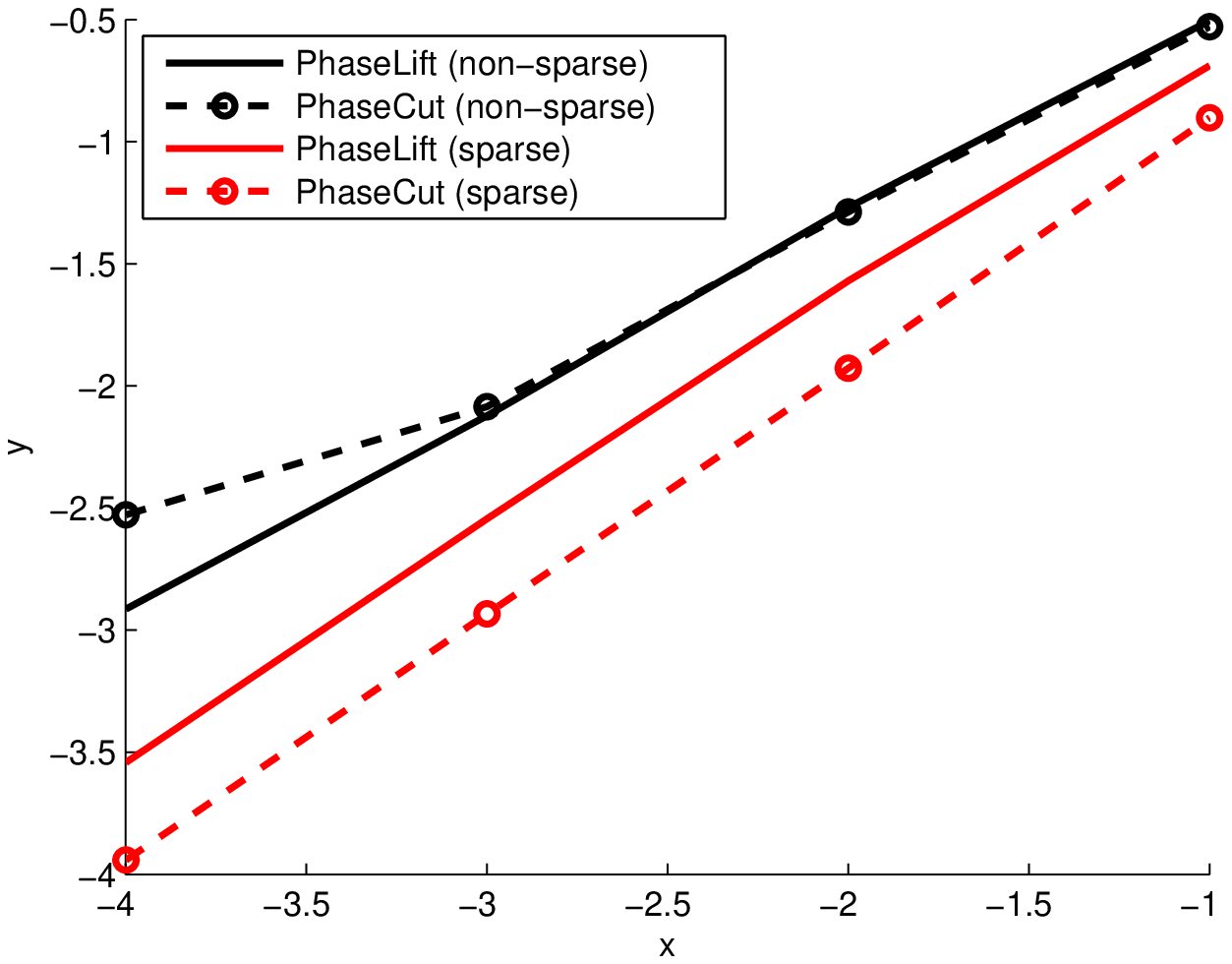}
\end{tabular}
\caption{{\em Left:} Mean performances of \ref{eq:ph-lift} and \ref{eq:ph-SDP}, followed by~\ref{alg:GS} iterations, for four Gaussian random illumination filters.  The $x$-axis represents the relative noise level, ${\|b_{\rm noise}\|_2}/{\|Ax\|_2}$ and the $y$-axis the relative error on the result, which is either $\epsilon(\tilde{x},x)$ or $\epsilon(|A\tilde{x}|,|Ax|)$. {\em Right:} Loglog plot of the relative error over the matrix reconstructed by \ref{eq:ph-lift} (resp. \ref{eq:ph-SDP}) when $A$ represents the convolution by five Gaussian filters. Black curves correspond to $Ax$ non-sparse, red ones to sparse $Ax$.}
\label{bruit_grf4_gs}
\end{figure}

%\begin{figure}
%\psfrag{y}[b][t]{Reconstruction error}
%\psfrag{x}[t][b]{Amount of noise}
%\psfrag{PhaseLift (for non-sparse measurements)}{\tiny PhaseLift (non-sparse)}
%\psfrag{PhaseCut (for non-sparse measurements)}{\tiny PhaseCut (non-sparse)}
%\psfrag{PhaseLift (for sparse measurements)}{\tiny PhaseLift (sparse)}
%\psfrag{PhaseCut (for sparse measurements)}{\tiny PhaseCut (sparse)}
%\includegraphics[width=.51\textwidth]{figures/bruit_sparsite_grf.eps}
%\caption{Loglog plot of the relative error over the matrix reconstructed by \ref{eq:ph-lift} (resp. \ref{eq:ph-SDP}) when $A$ represents the convolution by five Gaussian filters. Black curves correspond to $Ax$ non-sparse, red ones to sparse $Ax$.}
%\label{bruit_grf4_gs}
%\end{figure}

The impact of the sparsity of $b$ discussed in the last paragraph may seem irrelevant here: if $A$ and $x$ are independently chosen, $Ax$ is never sparse. However, if we do not choose $A$ and $x$ independently, we may achieve partial sparsity. We performed tests for the case of five Gaussian random filters, where we chose $x\in\C^{64}$ such that $(Ax)_k=0$ for $k\leq 60$. This choice has no particular physical interpretation but it allows us to check that the influence of sparsity in $|Ax|$ over \ref{eq:ph-lift} is not specific to the wavelet transform. Figure \ref{bruit_grf4_gs} displays the relative error over the reconstructed matrix in the sparse and non-sparse cases. If we denote by $X_{\rm pl}\in\C^{p\times p}$ (resp. $X_{\rm pc}\in\C^{n\times n}$) the matrix reconstructed by \ref{eq:ph-lift} (resp. \ref{eq:ph-SDP}), this relative error is defined by
\begin{gather*}
\epsilon = \frac{\|AX_{\rm pl}A^*-(Ax)(Ax)^*\|_2}{\|(Ax)(Ax)^*\|_2}
\hskip 1cm \mbox{(for PhaseLift)}\\
\epsilon = \frac{\|\diag(b)X_{\rm pc}\diag(b)-(Ax)(Ax)^*\|_2}{\|(Ax)(Ax)^*\|_2}
\hskip 1cm \mbox{(for PhaseCut)}
\end{gather*}
In the non-sparse case, both algorithms yield very similar error $\epsilon\approx 7 {\|b_{\rm noise}\|_2}/{\|Ax\|_2}$ (the difference for a relative noise of $10^{-4}$ may come from a computational artifact). In the sparse case, there are less phases to reconstruct, because we do not need to reconstruct the phase of null measurements. Consequently, the problem is better constrained and we expect the algorithms to be more stable. Indeed, the relative errors over the reconstructed matrices are smaller. However, in this case, the performance of \ref{eq:ph-lift} and \ref{eq:ph-SDP} do not match anymore: $\epsilon\approx 3 {\|b_{\rm noise}\|_2}/{\|Ax\|_2}$ for \ref{eq:ph-lift} and $\epsilon\approx 1.2 {\|b_{\rm noise}\|_2}/{\|Ax\|_2}$ for \ref{eq:ph-SDP}. This remark has no practical impact in our particular example here because taking a few \ref{alg:GS} iterations would likely make both methods converge towards the same solution, but it confirms the importance of accounting for the sparsity of $|Ax|$.

\section*{Acknowledgments} The authors are grateful to Richard Baraniuk, Emmanuel Cand\`es, Rodolphe Jenatton, Amit Singer and Vlad Voroninski for very constructive comments. In particular, Vlad Voroninski showed in \citep{Voro12a} that the argument in the first version of this paper, proving that PhaseCutMod is tight when PhaseLift is, could be reversed under mild technical conditions and pointed out an error in our handling of sparsity constraints. AA would like to acknowledge support from a starting grant from the European Research Council (project SIPA), and SM acknowledges support from ANR grant BLAN 012601.

\appendix
\section{Technical lemmas}
We now prove two technical lemmas used in the proof of Theorem~\ref{thm:stability}.
\begin{lemma}\label{lem:ecart_V_PC}
Under the assumptions and notations of Theorem~\ref{thm:stability}, we have
\begin{equation*}
\|V_{PC}^\sslash-(Ax_0)(Ax_0)^*\|_2> 2C\|Ax_0\|_ 2\|b_{\rm n,PC}\|_2
\end{equation*}
\end{lemma}
\begin{proof} We first give an upper bound of $\|V_{PC}-V_{PC}^\sslash\|_2$. We use the Cauchy-Schwarz inequality : for every positive matrix $X$ and all $x,y$, $|x^*Xy|\leq\sqrt{x^*Xx}\sqrt{y^*Xy}$. Let $\{f_i\}$ be an hermitian base of $\Range(A)$ diagonalizing $V_{PC}^\sslash$ and $\{g_i\}$ an hermitian base of $\Range(A)^\perp$ diagonalizing $V_{PC}^\perp$. As $\{f_i\}\cap\{g_i\}$ is an hermitian base of $\C^n$, we have
\begin{align}
\|V_{PC}-V_{PC}^\sslash\|_2^2&=\underset{i,i'}{\sum}|f_i^*(V_{PC}-V_{PC}^\sslash)f_{i'}|^2+\underset{i,j}{\sum}|f_i^*(V_{PC}-V_{PC}^\sslash)g_j|^2\nonumber\\
&\,\,\,\,\,+\underset{i,j}{\sum}|g_j^*(V_{PC}-V_{PC}^\sslash)f_i|^2+\underset{j,j'}{\sum}|g_j^*(V_{PC}-V_{PC}^\sslash)g_{j'}|^2\nonumber\\
&=2\underset{i,j}{\sum}|f_i^*(V_{PC})g_j|^2+\underset{i}{\sum}|g_i^*(V_{PC}^\perp)g_i|^2\nonumber\\
&\leq 2\underset{i,j}{\sum}|f_i^*(V_{PC})f_i\|g_j^*(V_{PC})g_j|+(\underset{i}{\sum} g_i^*(V_{PC}^\perp)g_i)^2\nonumber\\
&=2\Tr V_{PC}^\sslash\Tr V_{PC}^\perp+(\Tr V_{PC}^\perp)^2\nonumber\\
&\leq\left(\sqrt{2}\sqrt{\Tr V_{PC}^\sslash}\sqrt{\Tr V_{PC}^\perp}+\Tr V_{PC}^\perp\right)^2\label{eq:v_pc}
\end{align}

Let us now bound $\Tr V_{PC}^\perp$. We first note that $\Tr V_{PC}^\perp=\Tr((\idm-AA^\dag)V_{PC}(\idm-AA^\dag))=\Tr(V_{PC}(\idm-AA^\dag))=d_1(V_{PC},\mathcal{F})$ (according to lemma \ref{lem:dist_to_F}). Let $u\in\C^n$ be such that, for all $i$, $|u_i|=1$ and $(Ax_0)_i=u_i|Ax_0|_i$. We set $b=|Ax_0|+b_{\rm n,PC}$ and $V=(b\times u)(b\times u)^*$. As $V\in\herm_n^+\cap\mathcal{H}_b$ and $V_{PC}$ minimizes \eqref{eq:ph-cut-geom},
\BEAS
\Tr V_{PC}^\perp&=&d_1(V_{PC},\mathcal{F})\leq d_1(V,\mathcal{F})=d_1((Ax_0+b_{\rm n,PC}u)(Ax_0+b_{\rm n,PC}u)^*,\mathcal{F})\nonumber\\
&=&d_1((b_{\rm n,PC}u)(b_{\rm n,PC}u)^*,\mathcal{F})\nonumber\\
&\leq&\|(b_{\rm n,PC}u)(b_{\rm n,PC}u)^*\|_1=\Tr(b_{\rm n,PC}u)(b_{\rm n,PC}u)^*=\|b_{\rm n,PC}\|_2^2\label{eq:tr_v_pc_perp}
\EEAS
We also have $\Tr V_{PC}^\sslash=\Tr V_{PC}-\Tr V_{PC}^\perp$. This equality comes from the fact that, if $\{f_i\}$ is an hermitian base of $\Range(A)$ and $\{g_i\}$ an hermitian base of $\Range(A)^\perp$, then
\[
\Tr V_{PC}=\underset{i}{\sum}f_iV_{PC}f_i^*+\underset{i}{\sum}g_iV_{PC}g_i^*=\underset{i}{\sum}f_iV_{PC}^\sslash f_i^*+\underset{i}{\sum}g_iV_{PC}^\perp g_i^*=\Tr V_{PC}^\sslash+\Tr V_{PC}^\perp
\]
As $V_{PC}^\perp\succeq 0$, $\Tr V_{PC}^\sslash\leq\Tr V_{PC}=\||Ax_0|+b_{\rm n,PC}\|_2^2$ and, by combining this with relations~\eqref{eq:v_pc} and \eqref{eq:tr_v_pc_perp}, we get
\begin{align*}
\|V_{PC}-V_{PC}^\sslash\|_2&\leq \sqrt{2}\||Ax_0|+b_{\rm n,PC}\|_2\|b_{\rm n,PC}\|_2+\|b_{\rm n,PC}\|_2^2\\
&\leq \sqrt{2}\|Ax_0\|_2\|b_{\rm n,PC}\|_2+(1+\sqrt{2})\|b_{\rm n,PC}\|_2^2
\end{align*}
And, by reminding that we assumed $\|b_{\rm n,PC}\|_2\leq \|Ax_0\|_2$,
\begin{align*}
\|V_{PC}^\sslash-(Ax_0)(Ax_0)^*\|_2&\geq\|V_{PC}-(Ax_0)(Ax_0)^*\|_2-\|V_{PC}^\sslash-V_{PC}\|_2\\
&> D\|Ax_0\|_2\|b_{\rm n,PC}\|_2-\sqrt{2}\|Ax_0\|_2\|b_{\rm n,PC}\|_2-(1+\sqrt{2})\|b_{\rm n,PC}\|_2^2\\
&\geq (D-2\sqrt{2}-1)\|Ax_0\|_2\|b_{\rm n,PC}\|_2=2C\|Ax_0\|_2\|b_{\rm n,PC}\|_2
\end{align*}
which concludes the proof.
\end{proof}

\begin{lemma}\label{lem:ecart_b}
Under the assumptions and notations of Theorem~\ref{thm:stability}, we have $\|b_{\rm n,PL}\|_2\leq 2\|b_{\rm n,PC}\|$.
\end{lemma}
\begin{proof} Let $e_i$ be the $i$-th vector of $\C^n$'s canonical base. We set $e_i=f_i+g_i$ where $f_i\in\Range(A)$ and $g_i\in\Range(A)^\perp$.
\begin{align*}
V_{PC\,ii}&=e_i^*V_{PC}e_i\\
&=f_i^*V_{PC}^\sslash f_i+2\Re (f_i^*V_{PC}g_i)+g_i^*V_{PC}^\perp g_i\\
&=V_{PC\,ii}^\sslash+2\Re(f_i^*V_{PC}g_i)+V_{PC\,ii}^\perp
\end{align*}
Because $|f_i^*V_{PC}g_i|\leq\sqrt{f_i^*V_{PC}f_i}\sqrt{g_i^*V_{PC}g_i}=\sqrt{V_{PC\,ii}^\sslash}\sqrt{V_{PC\,ii}^\perp}$,
\begin{gather*}
(\sqrt{V_{PC\,ii}^\sslash}-\sqrt{V_{PC\,ii}^\perp})^2\leq V_{PC\,ii}\leq
(\sqrt{V_{PC\,ii}^\sslash}+\sqrt{V_{PC\,ii}^\perp})^2\\
\Rightarrow
\sqrt{V_{PC\,ii}^\sslash}-\sqrt{V_{PC\,ii}^\perp}\leq \sqrt{V_{PC\,ii}}\leq
\sqrt{V_{PC\,ii}^\sslash}+\sqrt{V_{PC\,ii}^\perp}
\end{gather*}
So
\begin{align*}
|b_{{\rm n,PL,}i}|&=|\sqrt{V_{PC\,ii}^\sslash}-|Ax_0|_i|\\
&\leq|\sqrt{V_{PC\,ii}^\sslash}-\sqrt{V_{PC\,ii}}|+|\sqrt{V_{PC\,ii}}-|Ax_0|_i|\\
&\leq\sqrt{V_{PC\,ii}^\perp}+b_{{\rm n,PC,}i}
\end{align*}
and, by \eqref{eq:tr_v_pc_perp},
\begin{align*}
\|b_{\rm n,PL}\|_2&\leq\|\left\{\sqrt{V_{PC\,ii}^\perp}\right\}_i\|_2+\|b_{\rm n,PC}\|_2\\
&=\sqrt{\Tr V_{PC}^\perp}+\|b_{\rm n,PC}\|_2\leq 2\|b_{\rm n,PC}\|_2
\end{align*}
which concludes the proof.
\end{proof}

%\section{Approximation of~\ref{eq:ph-SDP2}\label{app:ph-SDP2}}
%
%Here, we justify the affirmation of paragraph~\ref{ss:relaxation_tightness}: when $\gamma>0$ is small enough, solving
%\[\label{eq:ph-SDP2-simple}\tag{PhaseCutMod 2}
%\BA{ll}
%\mbox{minimize} & \Tr((M+\gamma B)U)\\
%\mbox{subject to} & \diag(U)=1,\, U \succeq 0,
%\EA\]
%approximately amounts to solving
%\BEQ\tag{PhaseCutMod}
%\BA{rll}
%SDP2(M) \triangleq & \mbox{min.} & \Tr(BU)\\
%& \mbox{subject to} & \Tr(MU)=0\\
%& & \diag(U)=1,\, U \succeq 0,
%\EA\EEQ
%
%\begin{lemma}
%For all $\gamma>0$, we note $U(\gamma)$ a solution of~\ref{eq:ph-SDP2-simple}. Let $(\gamma_n)$ be a sequence of positive numbers such that $\gamma_n\to 0$ and $(U(\gamma_n))_{n\in\N}$ converges to some limit $U_0$. Then $U_0$ is a solution of \ref{eq:ph-SDP2}.
%\end{lemma}
%\begin{proof}
%By continuity, $\diag(U_0)=1$ and $U_0\succeq 0$. Let $\alpha$ be the minimum of $\Tr(BU)$ in~\ref{eq:ph-SDP2}. Then
%\begin{equation*}
%\Tr((M+\gamma_nB)U(\gamma_n))\leq \gamma_n \alpha
%\end{equation*}
%As $\gamma_n\to 0$, $U(\gamma_n)\to U_0$ and $\Tr(MU(\gamma_n))\geq 0$ for all $n$, $\Tr(MU_0)=0$. Moreover:
%\begin{gather*}
%\gamma_n\Tr(BU(\gamma_n))\leq\Tr((M+\gamma_nB)U(\gamma_n))\leq\gamma_n\alpha\\
%\Rightarrow
%\gamma_n\Tr(BU_0)+o(\gamma_n)\leq\gamma_n\alpha
%\end{gather*}
%Dividing by $\gamma_n$ and letting $n$ go to $\infty$ yields $\Tr(BU_0)\leq \alpha$. As $\alpha$ is the minimal value of $\Tr(BU)$ in~\ref{eq:ph-SDP2}, $U_0$ is a solution of~\ref{eq:ph-SDP2}.
%\end{proof}

\small{\bibliographystyle{plainnat}\bibsep 1ex
\bibliography{/Users/aspremon/Dropbox/Research/Biblio/MainPerso.bib}}
%\bibliography{bib.bib}}
\end{document}

%% file: defs.tex
\newtheorem{theorem}{Theorem}[section]
\newtheorem{proposition}[theorem]{Proposition}
\newtheorem{definition}[theorem]{Definition}
\newtheorem{lemma}[theorem]{Lemma}
\newtheorem{corollary}[theorem]{Corollary}
\renewenvironment{proof}{\textbf{Proof.}}{\QED\bigskip}

\usepackage{algorithmic,algorithm}

\usepackage[square]{natbib}

\definecolor{ddarkbrown}{rgb}{0.5,0.2,0.05} \definecolor{bbluegray}{rgb}{0.05,0,0.5}

\newcommand{\BEAS}{\begin{eqnarray*}}
\newcommand{\EEAS}{\end{eqnarray*}}
\newcommand{\BEA}{\begin{eqnarray}}
\newcommand{\EEA}{\end{eqnarray}}
\newcommand{\BEQ}{\begin{equation}}
\newcommand{\EEQ}{\end{equation}}
\newcommand{\BIT}{\begin{itemize}}
\newcommand{\EIT}{\end{itemize}}
\newcommand{\BNUM}{\begin{enumerate}}
\newcommand{\ENUM}{\end{enumerate}}

\newcommand{\BA}{\begin{array}}
\newcommand{\EA}{\end{array}}

\newcommand{\reals}{{\mathbb R}}
\newcommand{\complex}{{\mathbb C}}

\newcommand{\symm}{{\mbox{\bf S}}}  % symmetric matrices
\newcommand{\herm}{{\mbox{\bf H}}}  % symmetric matrices

\newcommand{\Range}{\mbox{\textrm{range}}}

\newcommand{\range}{{\mathcal R}}

\newcommand{\Rank}{\mathop{\bf Rank}}

\newcommand{\Card}{\mathop{\bf Card}}
\newcommand{\Tr}{\mathop{\bf Tr}}
\newcommand{\diag}{\mathop{\bf diag}}

\newcommand{\idm}{\mathbf{I}}

\newcommand{\Expect}{\textstyle\mathop{\bf E}}

\newcommand{\QED}{~~\rule[-1pt]{6pt}{6pt}}

\renewcommand\Re{\operatorname{Re}}
\renewcommand\Im{\operatorname{Im}}

%% file: defs_p.tex
\usepackage{caption}
\usepackage{stmaryrd}

\newcommand{\C}{{\mathbb C}}
\newcommand{\R}{{\mathbb R}}

\newtheorem{remark}[theorem]{Remark}
\newtheorem*{lemma*}{Lemma}